\documentclass[11pt,a4paper]{article}
\usepackage{amsfonts,amsmath,amsfonts,graphicx,graphics,epsfig, caption, subcaption, verbatim}
\usepackage{enumerate}
\usepackage{color}
\usepackage{float}
\usepackage{multirow}
\newtheorem{theorem}{Theorem}

\newtheorem{lemma}{Lemma}

\newtheorem{proposition}{Proposition}

\newenvironment{proof}[1][Proof]{\noindent\textit{#1.} }{\ \rule{0.5em}{0.5em}}

\def\R{\mathbb{R}}
\def\L{\mathbb{L}}
\def\C{\mathbb{C}}
\begin{document}

\author{Cristina Butucea{\textdagger}, Rodrigue Ngueyep Tzoumpe* and Pierre Vandekerkhove \textdagger\textdaggerdbl \\
\it  \textdagger  Universit\'e Paris-Est \\
\it  LAMA (UMR 8050), UPEMLV \\
\it F-77454, Marne-la-Vall\'ee, France\\
\it * H. Milton Stewart School of Industrial Systems and Engineering\\
\it Georgia Institute of Technology\\
\it\textdaggerdbl \it UMI Georgia Tech - CNRS 2958,\\
\it School of aerospace\\
\it  Georgia Institute of Technology}
\title{Semiparametric topographical mixture models with symmetric errors}
\maketitle
\begin{abstract}
Motivated by the  analysis  of a Positron Emission Tomography (PET)   imaging data considered in Bowen et al. (2012), we introduce a semiparametric topographical mixture model
able to capture the characteristics of  dichotomous shifted  response-type experiments. We propose a local estimation procedure, based on the symmetry of the local noise, for  the proportion and locations functions  involved in the proposed  model.  We establish under mild conditions the   minimax properties  and asymptotic normality of our  estimators when   Monte Carlo simulations are conducted to examine  their  finite sample performance.  Finally a statistical  analysis of the PET imaging data in Bowen et al. (2012) is illustrated for the proposed method.

\end{abstract}
\vspace{0.5cm}

\bigskip

\newpage

\section{Introduction}
The model we propose to investigate in this paper  is a semiparametric  topographical mixture model able to capture  the characteristics of  dichotomous  shifted  response-type experiments such as the tumor data in Bowen et al. (2012, Fig. 4).  Let suppose that we visit at random the space $\R^d$ ($d \geq 1$) by sampling a sequence of  i.i.d. random variables ${\bf X}_i$,  $i = 1,...,n$, having common probability distribution function (p.d.f.) $\ell : \R^d \to \R_+$. For each ${\bf X}_i$ we observe an output response $Y_i$ whose distribution is a mixture model with probability parameters depending  on the design ${\bf X}_i$. For simplicity, let us consider first a mixture of two nonlinear regression model:
\begin{eqnarray}\label{modelrv}
Y_i=W ({\bf X}_i) (a({\bf X}_i) + \tilde \varepsilon_{1,i})
+(1-W({\bf X}_i)) (b({\bf X}_i) + \tilde \varepsilon_{2,i}),
\end{eqnarray}
where locations are $a,\, b :\R^d \to \R$, the errors $\{\tilde \varepsilon_{1,i},\, \tilde \varepsilon_{2,i}\}_{i=1,...,n}$ are supposed to be i.i.d with zero-symmetric common p.d.f. $f$.
The mixture in model (\ref{modelrv}) occurs according to the random variable $W({\bf x })$ at point ${\bf x}$, with probability $\pi : \R^d \to (0,1)$,
\begin{eqnarray*}
W({\bf x})
&=&\left\{\begin{array}{lll}1 & \mbox{with probability}&  \pi({\bf x}) , \\
0 & \mbox{with probability}&  1-\pi({\bf x}).
\end{array}\right.
\end{eqnarray*}
Moreover  we assume that, conditionally on the  ${\bf X}_i$'s, the $\{\tilde \varepsilon_{1,i},\, \tilde \varepsilon_{2,i} \}_i$'s and the $W({\bf X}_i)$'s are independent.
Such a model  is  linked to  the class of   Finite Mixtures of Regression (FMR), see Gr\"{u}n and Leisch (2006) for a good overview. Briefly,  statistical inference for the class of parametric FMR model was first considered by Quandt and Ramsey (1978) who proposed a  moment generating function based estimation method.
An EM estimating approach was proposed by De Veaux (1989) in the  two-component case. Variations of the latter approach were also considered in Jones and McLachlan
(1992) and Turner (2000). Hawkins et al. (2001) studied the estimation problem of
the number of components in the parametric FMR model using approaches derived from the
likelihood equation. In Hurn et al. (2003), the authors investigated  a Bayesian approach to
estimate the regression coefficients and also proposed an extension of the model in which
the number of components is unknown. Zhu and Zhang (2004) established the asymptotic
theory for maximum likelihood estimators in parametric FMR models. More recently,
St\"{a}dler et al. (2010) proposed an $\ell_1$-penalized method based on a Lasso-type estimator for
a high-dimensional FMR model with $d \geq n$.
As an alternative to parametric approaches to the estimation of a FMR model, some
authors suggested the use of more flexible semiparametric approaches. These approaches 
can actually be classified into  two groups: semiparametric FMR (SFMR)  of type I  and type II.
 The study of SFMR of type I comes from the seminal  work of Hall and Zhou (2003) in which $d$-variate semiparametric
mixture models of random vectors with independent components were considered. These authors
proved in particular that, for $d\geq 3$, we can  identify a two-component mixture model
without parametrizing the distributions of the component random vectors (Type I definition). To the best of
our knowledge, Leung and Qin (2006) were the first in estimating  a FMR model semiparametrically in that sense.
In the two-component case, they studied the case where the components
are related by Anderson (1979)'s exponential tilt model. Hunter and Young (2012) studied
the identifiability of an $m$-component  type I SFMR model and numerically investigated
a Expectation-Maximization (EM)  type algorithm for estimating its parameters. Vandekerkhove (2013) proposed
an M-estimation method for a two-component   semiparametric mixture of  linear regressions with
symmetric errors (type I) in which one component is known.  Bordes et al. (2013) revisited  the same model by establishing new moment-based
identifiability results from which they  derived  explicit $\sqrt n$-convergent  estimators. 
The study of type II SFMR models started with  Huang and Yao (2012)   who considered    a semiparametric  linear FMR model  with Gaussian noise  in which the mixing proportions are possibly covariates-dependent (Type II definition: parametric noises with mixing proportion and/or  noises' parameters functionally depending on covariates). They  established  also  the asymptotic normality  of their  local maximum likelihood estimator and investigated  a modified EM-type  algorithm. Huang et al. (2013) generalized the latter  work to nonlinear FMR with  possibly covariates-dependent  noises.  Toshiya (2013) considered  a  Gaussian  FMR model where the joint distribution of the response and the covariate (possibly functional)  is itself  modeled as a mixture. 
More recently Montuelle et al. (2013)  considered a penalized maximum likelihood approach for Gaussian FMR models  with logistic weights.

To improve the flexibility  of our  FMR model  (\ref{modelrv}) and address the study of models involving  design-dependent noises,  see radiotherapy application described  in Section  5,   we will  consider  a  slightly more general model:  
\begin{eqnarray}\label{modelrv1}
Y_i=W ({\bf X}_i) (a({\bf X}_i) + \varepsilon_{1,i}({\bf X}_i))
+(1-W({\bf X}_i)) (b({\bf X}_i) + \varepsilon_{2,i}({\bf X}_i)),
\end{eqnarray}
such that, given  $\left\{{\bf X} = {\bf x}\right\}$, the common  p.d.f. of the  $\varepsilon_{j,i}({\bf x})$, $j=1,2,$ denoted
 $f_{{\bf x}}$, is zero-symmetric. Note that the above model combines  type I  and type II properties since no parametric assumption is made about the noise and the mixing proportion,  along with the location parameters, are possibly design dependent.  Our model is still said {\it semiparametric} because, given $\left\{{\bf X} = {\bf x}\right\}$,  the vector $\theta({\bf x})=(\pi({\bf x},)a({\bf x}),b({\bf x}))$ will be  viewed as an Euclidean parameter to be estimated.\\

\noindent {\it Examples of design-point noise dependency.}  
\begin{enumerate}[i)]
\item (Topographical scaling) The most natural transformation is probably when   considering  a topographical scaling of the errors,  
with $\sigma : \R^d \to \R^*_+$, such that $\varepsilon_{j,i}({\bf X}_i) = \sigma({\bf X}_i) \tilde \varepsilon_{j,i} $, $j=1,2$, where the $\tilde \varepsilon_{j,i}$'s are similar to those involved  in (\ref{modelrv}). The conditional p.d.f given $\left\{{\bf X} = {\bf x}\right\}$ is defined  by
\begin{eqnarray}\label{topovar}
f_{{\bf x}} (y) = \frac 1{\sigma(\bf x)} f\left( \frac{y}{\sigma(\bf x)}\right), \quad y \in \R.
\end{eqnarray}
Indeed,  if $f$ is zero-symmetric then  the errors' distribution inherits trivially  the same symmetry property.
\item (Zero-symmetric varying  mixture) Another useful  example could be the varying mixing proportion mixture model of $r$ zero-symmetric distributions. For $k=1,\dots,r$,  we consider proportion functions $\lambda_k : \R^d \to (0,1)$ with $\sum_{k=1}^r \lambda_k ({\bf x}) = 1$ for all ${\bf x}\in \R^d$. The conditional p.d.f given $\left\{{\bf X} = {\bf x}\right\}$ is defined  by
\begin{eqnarray*}\label{topovar1}
f_{{\bf x}} (y) = \sum_{k=1}^r \lambda_k({\bf x}) f_k(y), \quad y \in \R,
\end{eqnarray*}
where the $f_k$ functions  are zero-symmetric p.d.f.'s. 
\item (Antithetic location model)  Consider a location function $\mu : \R^d \to \R$ and $f$ any arbitrary  p.d.f. The conditional p.d.f given $\left\{{\bf X} = {\bf x}\right\}$ is defined  by
\begin{eqnarray*}\label{topovar2}
f_{{\bf x}} (y) =\frac{1}{2} f(y-\mu({\bf x}))+\frac{1}{2} f(-y+\mu({\bf x})), \quad y \in \R,
\end{eqnarray*}
and also  results into a zero-symmetric p.d.f.
\end{enumerate}
Note that any  combination of the above situations could be  considered  in model (\ref{modelrv1}) free from specifying any parametric family (provided the resulting zero-symmetry hold). This last remark reveals, according to us, the main strength of our model in the sense  that it  could prove to be  a  very  flexible  exploratory tool  for  the analysis of   shifted  response-type experiments. Our paper is organized as follows.
Section 2 is devoted to a detailed description of our estimation method, while Section 3 is concerned with  its asymptotic properties.  The finite-sample performance of the proposed estimation method is studied for various scenarios through Monte Carlo experiments in Section 4.  In Section 5 we propose to analyze the  Positron Emission Tomography (PET) imaging data considered in Bowen et al. (2012). Finally Section 6 is devoted to auxiliary results and  main proofs.

\section{Estimation method}

\noindent Let us define the joint density of  couples $(Y_i,{\bf X}_i)$, $i=1,\dots,n$,  designed from model (\ref{modelrv1}):
\begin{eqnarray}\label{jointdensity}
g(y,{\bf x})=[\pi({\bf x })f_{\bf x}(y-a({\bf x }))
+(1-\pi({\bf x}))f_{\bf x}(y-b({\bf x}))]\ell({\bf x}),
\quad (y,{\bf x
})\in \mathbb R^{d+1},
\end{eqnarray}
while the conditional density of $Y$ given $\left\{{\bf X} = {\bf x}\right\}$ (denoted for simplicity $Y/{\bf X}={\bf x}$) is
\begin{eqnarray}\label{conddensity}
g_{{\bf x}}(y) = g(y,{\bf x}) / \ell({\bf x}) = \pi({\bf x })f_{\bf x}(y-a({\bf x }))
+(1-\pi({\bf x}))f_{\bf x}(y-b({\bf x})).
\end{eqnarray}

\noindent We are interested in estimating the parameter $\theta_0
 =  \theta({\bf x}_0
)=(\pi({\bf x}_0
), a({\bf x}_0
), b({\bf x}_0
))$ at some fixed point ${\bf x}_0
$ belonging to the interior of the support of $\ell$ ($\ell({\bf x}_0)>0$), denoted  $supp(\ell)$.  For simplicity and identifiability matters, we will suppose that  $\theta_0
$ belongs to the interior of the parametric space  $\Theta=[p,P]\times \Delta$, where $0< p \leq P < 1/2$ and $\Delta$ denotes a compact set of $\R^2 \backslash \{(x,x): x\in \R\}$.\\


\subsection{Mixture of regression functions as an inverse problem}

\noindent We see in formula (\ref{conddensity}), that the conditional density of $Y$ given 
$\left\{{\bf X} = {\bf x}\right\}$ can be viewed  as a mixture of the errors distribution $f_{\bf x}$ given $\left\{{\bf X} = {\bf x}\right\}$ with locations $(a({\bf x}), b({\bf x}))$ and mixing proportion $\pi({\bf x})$. Mixture of populations with different locations is a well known  inverse problem. Our inversion procedure is here  based on the Fourier transform of the conditional density  $g_{{\bf x}}(y) $ 
of $Y/{\bf X} = {\bf x}$. 
If the p.d.f. $g_{{\bf x}}$ belongs to $\mathbb{L}_1 \cap \mathbb{L}_2$,  define 
$g^*_{{\bf x}}(u) = \int \exp[i uy] g_{{\bf x}}(y) dy $ for all  $u\in \R$, and observe that
\begin{eqnarray}\label{TFg}
g^*_{{\bf x}}(u)
&=&\left (\pi({\bf x}) e^{iu a({\bf x})}+(1-\pi({\bf x})) e^{iu b({\bf x})}\right )f^*_{{\bf x}}(u), \quad u\in \mathbb R \nonumber .
\end{eqnarray}
Let us denote, for all $(t,u)=(\pi, a, b,u)\in \Theta\times \R$,  
\begin{eqnarray}\label{M}
M(t,u):=\pi e^{iu a}+(1-\pi) e^{iu b}.
\end{eqnarray}
 Note that $|M(t,u)| \in [1-2P,1]$ for all $(t,u)\in \Theta\times \R$.
Then, we have
\begin{eqnarray*}\label{local_Fourier}
g^*_{{\bf x}}(u) & = & M(\theta({\bf x}),u)f^*_{{\bf x}_0}
(u).
\end{eqnarray*}
 Let us fix ${\bf x}_0
\in supp(\ell)$  such that $\theta({\bf x}_0
) $ belongs  to the interior of  $\Theta$, denoted   $\stackrel{\circ}{\Theta}$. Noticing  that the p.d.f. $f_{{\bf x}_0
}$ is zero-symmetric we  therefore have that  $f_{\bf x_0}^*(u)\in\R$, for all $u\in \R$. 
 If $t$ belongs to $\Theta$, we prove in the next theorem the {\it picking} property
\begin{eqnarray*}
\Im \left( \frac{g^*_{{\bf x}_0}(u)}{M(t,u)}\right) =0 \text{ for all } u \in \R, \text{ if and only if } t=\theta({\bf x}_0),
\end{eqnarray*}
where $\Im: \C\rightarrow \R$ denotes the imaginary part of a complex number.
This result allows us to build  a {\it contrast}  function  for the parameter $t\in \Theta$:
\begin{equation}\label{contrast}
S(t) := \int \Im^2 \left( \frac{g^*_{{\bf x}_0
}(u)}{M(t,u)}\right) \ell^2({\bf x}_0
) w(u)du.
\end{equation}
The function $w:\R^d \to \R_+$ is a bounded p.d.f. which helps in computing the integral via Monte-Carlo method and solves integrability issues.\\

\noindent {\it Remark.} The idea of using  Fourier transform in order to solve the inverse mixture  problem was introduced in Butucea and Vandekerkhove (2013) for density models. In the regression models we deal with the conditional density of $Y/{\bf X}={\bf x}_0
$. This has no incidence on the identifiability of the model  but changes dramatically the behavior of the estimators as we shall see later on.
\\

\noindent We prove in the following theorem   that our model is identifiable and that $S(t)$ defines a contrast on the parametric space $\Theta$.

\begin{theorem}{(Identifiability and contrast property)}
Consider  model (\ref{modelrv1}) provided with   $f_{{\bf x}}(\cdot) \in \L_1$ for all ${\bf x}\in \R^d$. For a fixed point ${\bf x}_0
$ in the interior of the support of $\ell$,  we assume that $f_{{\bf x}_0
}(\cdot)$ is zero-symmetric and that $\theta_0
=\theta({\bf x}_0
)$ is an  interior point  of $\Theta$. Then we have the following properties:
\begin{enumerate}[i)]
\item The collection of scalar parameters $\theta_{0} =(\pi({\bf x}_0
), a({\bf x}_0
), b({\bf x}_0
))$ and the function $f_{{\bf x}_0
}(\cdot)$ are identifiable.

\item  The function $S$ in (\ref{contrast}) is a contrast function, i.e. for all $t \in \Theta$, $S(t)\geq 0$ and $S(t)=0$ if and only if $t=\theta_0
$.
\end{enumerate}
\end{theorem}
\begin{proof} The proofs of i) and ii) are respectively similar to the proof of  Theorem 1 and Proposition 1 in Butucea and Vandekerkhove (2013), replacing  $f^*(\cdot)$ and $g^*(\cdot)$
by $f_{{\bf x}_0
}^*(\cdot)$ and $g_{{\bf x}_0
}^*(\cdot)$, and noticing that $\ell({\bf x}_0
)$ is bounded away from zero. Follows also Theorem~2.1 in Bordes et al. (2006).
\end{proof}

\bigskip

\noindent {\it Remark.} For mixture models with higher number of components, i.e.
\begin{eqnarray*}
Y_i= \sum_{j=1}^J W_j ({\bf X}_i) (\gamma_j({\bf X}_i) + \varepsilon_{j,i}({\bf X}_i)),\quad i=1,\dots,n,
\end{eqnarray*}
where  $(W_1({\bf x}),..., W_J({\bf x}))$ are distributed according to a $J$-components ($J>2$)  multinomial distribution with parameters
$(\pi_1({\bf x}),...,\pi_J({\bf x})),$ and noises $(\varepsilon_{j,i})$, $j=1,\dots,J$,   i.i.d. according to $f_{\bf x}$, 
we assume that there exists a compact set $\Psi\subset ]0,1[^{J-1}\times \R^J$ of parameters
$(\pi_1({\bf x}),...,\pi_{J-1}({\bf x}),\gamma_1({\bf x}),...,\gamma_J({\bf x}))$ where the model is {\it identifiable}, see Hunter et al. (2007, Section 2). Note that the  3-components mixture model has been  studied closely  in  Bordes et al. (2006) and  Hunter et al. (2007) where sufficient identifiability  conditions were given. The case where $d>3$ is more involved for full description  and it is still an open question. In this setup, the  estimation procedure described  hereafter can be adapted  over the parameter  space  $\Psi$ with analogous results.

\subsection{Estimation procedure}

In order to build  an estimator of the contrast $S(t)$ defined in (\ref{contrast}), a local smoothing has to be performed in order to extract the information that the random design $X_1,...,X_n$ brings to the knowledge of  the conditional law of $Y/{\bf X}={\bf x}_0
$. We use a  kernel smoothing approach, but local polynomials or wavelet methods could also be employed. This smoothing is a major difference with respect to the density model considered in Butucea and Vandekerkhove (2013) and all the rates will depend on the smoothing parameter applied to the kernel function.

We choose a kernel function $K:\R^d\to \R$ belonging to $\mathbb{L}_1$ and to $\mathbb{L}_4$ and some bandwidth parameter  $h>0$ to be described later on. For ${\bf x}_0
\in supp(\ell)$ fixed, we denote
\begin{equation}\label{eq:Zk}
Z_k(t,u,h) := \left( \frac{e^{iuY_k}}{M(t,u)} - \frac{e^{-iuY_k}}{M(t,-u)} \right) K_h({\bf X}_k-{\bf x}_0
), \text{ where } K_h({\bf x}) := \frac 1{h^d} K\left( \frac {{\bf x}}{h}\right).
\end{equation}
The empirical contrast of $S(t)$ is defined by
\begin{equation}\label{estim_contrast}
S_{n}(t) = -\frac 1{4n (n-1)} \sum_{j\ne k, j,k=1}^n
\int Z_k(t,u,h) Z_j(t,u,h) w(u) du,
\end{equation}
where $w:\mathbb{R}\to \mathbb{R}^*_+$ is a bounded p.d.f., having a finite  moment of order 4, i.e. $\int u^4 w(u)du < \infty$.
From  this empirical contrast we then define the estimator
\begin{equation}\label{estim_theta}
\hat \theta_n = \arg\inf_{t \in \Theta} S_n(t),
\end{equation}
of $\theta_0
 = \theta({\bf x}_0
)$. We shall study successively the properties of $S_n(t)$ as an estimator of $S(t)$ and deduce consistency and asymptotic normality of $\hat \theta_n$ as an estimator of $\theta_0
$.\\

\noindent{\it Estimation methodology for  $f_{\bf x_0}$}. For the estimation of the local noise density $f_{{\bf x}_0}$ we suggest to consider the natural smoothed version of  the plug-in density  estimate given in Butucea and Vandekerkhove (2013, Section 2.2).

Let us denote by $\varphi({\bf x},y)=\ell({\bf x})f_{\bf x}(y)$. We plug $\hat \theta_n$  in the natural  smoothed nonparametric kernel estimator of $\varphi({\bf x},y)$  deduced from (\ref{TFg}), whenever the unknown parameter $\theta_0$ is required.
For ${\bf x}_0$ fixed, we consider the Fourier transform of the resulting estimator of  $\varphi({\bf x}_0,y)$. This   procedure gives, in Fourier domain,
\begin{equation*}\label{kernelfour}
 \varphi_{{\bf x}_0,n}^* (u) = \frac 1n \sum_{k=1}^n \frac{Q^*(h_{1,n} u)e^{iuY_k}}{M(\hat \theta_{n},u)}K_{h_{2,n}}({\bf X_k}-{\bf x}_0),
\end{equation*}
where $Q$ is a univariate  kernel ($\int Q = 1$ and $Q \in \mathbb{L}_2$) and $(h_{1,n},h_{2,n})$ are  bandwidth parameters   properly chosen.
Note that $G_n^*(u):=Q^*(h_{1,n} u)/M(\hat \theta_{n},u)$ is in $\L_1$ and $\L_2$ and has an inverse
Fourier transform which we denote by $G_n(u/h_{1,n})/h_{1,n}$. Therefore, the estimator of $\varphi({\bf x_0},y)$ is
\begin{equation*}\label{kernelest}
\varphi_n({\bf x}_0,y) = \frac 1{nh_{1,n}} \sum_{k=1}^n G_n\left( \frac{x-X_k}{h_{1,n}}\right)K_{h_{2,n}}({\bf X_k}-{\bf x}_0).
\end{equation*}
Finally the estimator of $f_{{\bf x}_0}$ is obtained by considering
\begin{eqnarray}\label{local_dens}
\hat f_{{\bf x}_0} (y)= \frac{f_n(y|{\bf x}_0){\mathbb I}_{f_n(y|{\bf x}_0)\geq 0}}{\int_\R f_n(y|{\bf x}_0){\mathbb I}_{f_n(y|{\bf x}_0)\geq 0}dy } ,\quad \mbox{where}~~ f_n(y|{\bf x}_0)=\frac{\varphi_n({\bf x}_0,y)}{\ell_n({\bf x}_0)}.
\end{eqnarray}
where $\ell_n({\bf x}_0)=\frac 1{n} \sum_{k=1}^n K_{h_{2,n}}({\bf X_k}-{\bf x}_0)$.
The asymptotic properties of this local density   estimator are not established yet but we strongly guess   that the bandwidth conditions  required to prove 
its convergence and classical convergence rate are similar to those found in the conditional density estimation literature,  see  Brunel et al. (2010) or  Cohen and Le Pennec (2012).
\section{Performance of the method}\label{Perf}
We give upper bounds for the mean squared error of $S_n(t)$. We are interested in consistency and asymptotic normality of $\hat \theta_n$ and this requires some small amount of smoothness $\alpha \in (0,1]$ for the  p.d.f.  of the errors and for the functions $\pi,\, a$ and $b$.  From now on, $\|v\|$ denotes the Euclidean norm of vector $v$. Recall that a function $F$ is Lipschitz $\alpha$-smooth if it belongs to the following class
$$
L(\alpha, M) = \left\{ F: \mathbb{R}^d \to \mathbb{R}, |F({\bf x}) - F({\bf y})| \leq M \|{\bf x}-{\bf y}\|^\alpha, ~({\bf x}, \, {\bf y})\in \R^d\times \R^d\right\},
$$
for $\alpha \in (0,1] $ and $M>0$. 

\vspace{0.2cm}

\noindent {\bf A1}. We assume that the functions $\pi, \, a,\, b, \, \ell$ are Lipschitz $\alpha$-smooth  with constant $M>0$. 
\vspace{0.2cm}

\noindent {\it Remark.} We may actually suppose that the functions appearing in our model have different smoothness parameters, but the rate will be governed by the smallest smoothness parameter.

An important consequence of this assumption is that the density $\ell$ is uniformly bounded by some constant depending only on $\alpha$ and $M$, i.e. $\sup_{\ell \in L(\alpha, M)} \|\ell\|_\infty < \infty$.

\vspace{0.2cm}

\noindent {\bf A2.} Assume that  $f_{{\bf x}}(\cdot) \in \L_1 \cap \L_2$ for all ${\bf x} \in \R^d$.
In addition, we require that  there exists a $w$-integrable function $\varphi$  such that 
$$
|f^*_{\bf x}(u)-f^*_{{\bf x}'}(u)|\leq \varphi(u) \|{\bf x}-{{\bf x}'}\|^\alpha,\quad ({\bf x},{\bf x}')\in \R^d\times \R^d, ~u\in \R.
$$

\noindent {\it Remark.} Note that for the scaling model (\ref{topovar}), if $f$ is the ${\mathcal N}(0,1)$ p.d.f. and $\sigma(\cdot)$ is bounded  and Lipschitz $\alpha$-smooth, we have:
$$
|f^*_{\bf x}(u)-f^*_{{\bf x}'}(u)|\leq  \frac{u^2}{2} |\sigma^2({\bf x})-\sigma^2({{\bf x}'})|\leq  \frac{u^2}{2} \|{\bf x}-{{\bf x}'}\|^\alpha.
$$

\noindent {\bf A3.} We assume that the kernel $K$ is such that $\int |K| < \infty$,  $\int K^4 < \infty$ and that it satisfies also the moment condition
$$
\int \|{\bf x}\|^\alpha |K({\bf x})| d{\bf x} < \infty.
$$

\noindent {\bf A4.} The weight function $w$ is a p.d.f. such that
$$
\int (u^4+\varphi(u)) w(u) du<\infty.
$$

\noindent {\it Remark.} We may suppose that the smoothness $\alpha >1$. In that case, the class $L(\alpha,M)$ consists of all functions $F$ with bounded derivatives up to order $k$, where $\alpha=k+\beta$, $k\in \mathbb{N}$ and $\beta \in (0,1]$. Moreover, for all  multi-index $j=(j_1,...,j_d) \in \mathbb{N}^d$ such that $|j|=k$ where $|j|=j_1+...+j_d$,  we have
$$
|F^{(j)}({\bf x}) - F^{(j)}({\bf y})| \leq M \|{\bf x}-{\bf y}\|^\beta, \quad ({\bf x},{\bf y})\in \R^d\times \R^d.
$$
The following results will hold true under the additional assumption on the kernel (see {\bf A3}): 
$\int {\bf x}^j K({\bf x}) d{\bf x} = 0$, for all $j$ such that $|j|\leq k$.

\begin{proposition}{}\label{mse_contrast}
For each $t \in \Theta$ and  ${\bf x}_0
\in supp(\ell)$ fixed, suppose $\theta_0
  \in \stackrel{\circ}{\Theta}$ and that assumptions  {\bf A1-A4}  hold.  Then, the empirical contrast function $S_n(\cdot)$ defined in (\ref{estim_contrast}) satisfies 
$$
 E\left[ \left( S_{n}(t) - S(t)\right)^2\right]
\leq C_1 h^{2\alpha } + C_2 \frac 1{nh^d} ,  
$$
if $ h\to 0$ and $nh^d \to \infty$  as $n \to \infty$,  where constants $C_1, \, C_2$ depend on $\Theta$, $K$, $w$, $\alpha$ and $M$ but are free from  $n,\, h, \, t$ and ${\bf x}_0
$.  
\end{proposition}

\begin{theorem}{\bf (Consistency)}\label{cons}
Let suppose that assumptions of Proposition \ref{mse_contrast} hold and  consider model (\ref{modelrv1}) with likelihood given by (\ref{jointdensity}). If  the p.d.f $f_{{\bf x}_0
}$ is   zero-symmetric,   then  the estimator $\hat \theta_{n}$  defined in (\ref{estim_contrast}-\ref{estim_theta})  converges in probability to 
 $\theta({\bf x}_0
)=\theta_0
$  if  $h\rightarrow 0$  and $nh^d\rightarrow \infty$ as $n\rightarrow \infty$.

\end{theorem}

\vspace{0.2cm}
The following theorem establishes the asymptotic normality of the estimator $\hat \theta_n$ of $\theta_0
$. Recall that $\theta_0
 = \theta({\bf x}_0
)$ belongs to $\Theta$ and that there exists $l>0$ such that $\ell({\bf x}_0
) \geq l$. We see that the local smoothing with bandwidth $h>0$ deteriorates the rate of convergence to $\sqrt{nh^d}$ instead of $\sqrt{n}$ for the density model. 
In the asymptotic variance we will use the following notation:
\begin{equation}\label{eq:Sigma1}
\dot J(\theta_0
,u): = \Im
 \left( - \frac{ \dot M(\theta_0
,u)}{M(\theta_0
,u)} \right)  f^*_{{\bf x}_0
}(u) \ell({\bf x}_0
),
\end{equation}
and
\begin{equation}\label{eq:Sigma2}
 V(\theta_0
,u_1,u_2) := \int \left( \frac{e^{iu_1 y}}{M(\theta_0
,u_1)} - \frac{e^{-iu_1 y}}{M(\theta_0
,-u_1 )}\right)
\left( \frac{e^{iu_2 y}}{M(\theta_0
,u_2 )} - \frac{e^{-iu_2 y}}{M(\theta_0
,-u_2 )}\right)
g_{{\bf x}_0
}(y) dy,
\end{equation}
where the  function $M(\cdot,\cdot)$ is defined in (\ref{M}). Note that $\dot J(\theta_0
,\cdot)$ is uniformly bounded by some constant and that $V$ is well defined for all $(u_1,u_2)\in \R\times \R$ and also uniformly bounded by some constant.
 
\begin{theorem}\label{asymptotic_normality}{\bf (Asymptotic normality)} Suppose that assumptions of Theorem \ref{cons} hold. The estimator $\hat \theta_n$ of $\theta_0
$ defined by (\ref{estim_contrast}-\ref{estim_theta}), with $h\to 0$ such that $n h^d \to \infty$ and such that $h^{2\alpha +d} = o(n^{-1})$, as $n\to \infty$, is asymptotically normally distributed:
$$
\sqrt{n h^d} (\hat \theta_n - \theta_0
) \to N(0, \mathcal{S}) \quad {\mbox{in distribution}},
$$
where $\mathcal{S} = \frac 14 \mathcal{I}^{-1} \Sigma \mathcal{I}$,  with 
$$\mathcal{I}=-\frac{1}{2}\int \dot J(\theta_0
,u) \dot J(\theta_0
,u)^{\top} dw(u),$$ 
and
$$
\Sigma:= \int \int \dot J(\theta_0
,u_1) \dot J^\top(\theta_0
,u_2) V(\theta_0
, u_1, u_2) w(u_1) w(u_2) du_1 du_2,
$$
for $\dot J$ defined in (\ref{eq:Sigma1}) and $V$ in (\ref{eq:Sigma2}).
\end{theorem}

The above results show that our  estimator of $\theta_0
$ behaves like any nonparametric pointwize estimator. This is indeed the case and we provide in the next theorem the best achievable convergence rates uniformly over the large set of functions involved in our model, see assumptions {\bf A1}-{\bf A2}.  For length matters,  we will just provide some hints of  proof of the next theorem. 

\begin{theorem}\label{minimax}{\bf (Minimax rates)} Suppose  {\bf A1}-{\bf A4} and consider  ${\bf x}_0
\in supp(\ell)$ fixed  such that $\ell({\bf x}_0
) \geq L_* >0$ for all $\ell\in L(\alpha,M)$ and $\theta_0
=\theta({\bf x}_0
)  \in \stackrel{\circ}{\Theta}$.
The estimator $\hat \theta_n$ of $\theta_0
$ defined by (\ref{estim_contrast}-\ref{estim_theta}), with $h \asymp n^{-1/(2\alpha +d)}$, as $n\to \infty$, is such that 
$$
\sup E[ \|\hat \theta_n - \theta_0
 \|^2] \leq C n^{-\frac{2\alpha}{2\alpha+d}}, 
$$
where the supremum is taken over all the functions ${\pi, a,b, \ell}$ and $f^*$ checking  assumptions  {\bf A1}-{\bf A2}. 
Moreover, 
$$
\inf_{T_n}\sup E[ \|T_n - \theta_0
\|^2] \geq c n^{-\frac{2\alpha}{2\alpha+d}},
$$
where $C, \, c>0$ depend only on $\alpha, M, \Theta, K$ and $w$, and the infimum is taken over the set of all the  estimators $T_n$ (measurable function of the observations $(X_1,\dots, X_n))$  of $\theta_0$.
\end{theorem}
\noindent{\it Proof hints.} Throughout the proofs of the previous results we learn that the estimator $\hat \theta_n$ of $\theta_0
$, behaves asymptotically as $\dot S_n(\theta_0)$ which is a $U$-statistic with a dominant term whose bias is of order $h^{2\alpha}$ and whose variance is smaller than $C_2 (nh^d)^{-1}$. The bias-variance compromise will produce an optimal choice of the bandwidth $h$ of order $ n^{-1/(2\alpha +d)}$ and a rate $n^{-\frac{2\alpha}{2\alpha+d}}$. It is the optimal rate for estimating a Lipschitz $\alpha$-smooth regression function at a fixed point  and the optimality results in the previous theorem are a consequence of the general nonparametric problem, see Stone (1977),  Ibragimov and  Has'minski (1981) and  Tsybakov (2009).


\section{Practical behaviour}
\subsection{Algorithm}
We describe below the initialization scheme and the optimization method used to determine the estimates of the locations $a({\bf x}_k)$, $b({\bf x}_k)$ and the weight functions $\pi({\bf x}_k)$ for a  fixed sequence of  testing points $\displaystyle \left\{{\bf x}_k,~k=1,\dots, K\right\}$. To simply differentiate these testing points from  the design data points we will allocate specifically the index $k$ for the numbering  of the testing points  and the index $i$ for the numbering of the  dataset points, i.e.   $\displaystyle \left\{({\bf x}_i,y_i), ~i=1,\dots, n\right\}$. \\

\noindent {\it  Initialization}

\begin{enumerate}
\item For each design data point ${\bf x}_i$, $ i=1,\dots,n,$  fit a kernel regression smoothing  $\bar{m}({\bf x}_i)$ with local bandwidth $\bar h_{{\bf x}_i}$. The \texttt{ R} package  \texttt{lokerns}, see   Herrmann (2013),  can be used.
\item Classify each data point $({\bf x}_i, y_i)$,  $i=1,\dots, n$ according to:  if $y_i > \bar {m}({\bf x}_i)$ classify $({\bf x}_i,y_i)$ in group 1 associated with location $a(\cdot)$, otherwise classify it in  group 2  associated with $b(\cdot)$.
\item  For each ${\bf x}_k$, $k=1,\dots,K$,  obtain  initial value $\bar{a}({\bf x}_k)$, respectively  $\bar{b}({\bf x}_k)$, by fitting  a kernel regression smoothing  based on the  observations $({\bf x}_i,y_i)$ , $i=1,\dots, n$,  previously classified in group 1 with local bandwidth $\bar h_{1,{\bf x}_k}$, respectively  in group 2 with local bandwidth $\bar h_{2,{\bf x}_k}$.
\item Compute  the local bandwidth  $h_{{\bf x}_k}= \min(\bar h_{1,{\bf x}_k}, \bar h_{2,{\bf x}_k})$.
\item Fix an arbitrary single value $\bar \pi$  for all  the $\pi({\bf x}_k)$'s.
\end{enumerate}

\noindent{\it Estimation}

\begin{enumerate}
\item Generate  one  $w$-distributed i.i.d sample $(U_r)$, $r=1,\dots,N$  dedicated to the  pointwize Monte Carlo estimation of $S_n(t)$ defined by:
\begin{equation*}\label{estim_contrast_MC}
S^{MC}_{n}(t) = -\frac 1{4n (n-1)N} \sum_{j\ne k, j,k=1}^n
\sum_{r=1}^N Z_k(t,U_r,h) Z_j(t,U_r,h).
\end{equation*}
In  the Sections 4.2 and 5, we will consider $N=n$ and $w$ the p.d.f. corresponding to the mixture $0.1*{\mathcal N}(0,1)+0.9*{\mathcal U}_{[-2,2]}$.
\item Compute the minimizer $\hat \theta({\bf x}_k)=(\hat \pi({\bf x}_k),\hat a({\bf x}_k), \hat b({\bf x}_k))$  of  $S^{MC}_n(\cdot)$  evaluated  at each  point ${\bf x}_0
={\bf x}_k$, by using the starting values 
$(\bar \pi,\bar a({\bf x}_k), \bar b({\bf x}_k))$ and the local bandwidth $h_{{\bf x}_k}$.
\end{enumerate}
In our simulations,  the above minimization will  be, contrarily to the theoretical requirements,  deliberately done over a non-constrained space, i.e. generically $\theta(\cdot)\in [0.05,0.95]\times [A,B]^2$, with  $A<B$. Our goal is to analyze  experimentally if a performant  initialization procedure is able to prevent from spurious  phenomenons like the label switching or component merging occurring when $\pi({\bf x}_0)$ is close to $0.5$. This kind of information is actually very  relevant  to interpret correctly  some cross-over effects as the one we will observe in 
Fig. 6 (a). Note that   other initialization methods can be figured out. We can for instance use, similarly to Huang et al. (2013),  a mixture of polynomial regressions with constant proportions and variances to pick  initial values   $\bar{a}(x)$ and $\bar{b}(x)$, or   the  \texttt{R} package \texttt{flexmix}, see  Gruen et al. (2013),    that implements a general framework for finite mixture of regression models based on  EM-type  algorithms (we selected this latter approach for the analysis of  radiotherapy application in Section 5).

\subsection{Simulations}
In this section, we propose to  measure the performances of our  estimator $\hat \theta_n(\cdot)$ over a testing sequence  $\displaystyle\left\{ {\bf x}_k=k/K\right\}$, $k=1,\dots,K=20$. Given that in the simulation  setting the true function  $ \theta(\cdot)$ is known, we can compute, similarly to Huang et al. (2013),   the  Root Average Squared Errors (RASE) of our estimator.  
To this end we generate $M=100$ datasets  $({\bf X}_i^{[z]}, Y_i^{[z]})_{1\leq i\leq n}$, $z=1, \dots, M$  of sizes $n$= 400, 800, 1200,  for each of the scenario described below and, for each scalar parameter $s=a,~b,~\pi$,  denote by $RASE^{[z]}_{s}$ the RASE performance associated to the $z$-th dataset, defined by $RASE^{[z]}_{s}=(1/K\sum_{k = 1}^{K}   R_s^{[z]}(k))^{1/2}$, where $R_s^{[z]}(k) = \left(\hat{s}^{[z]}({\bf x}_k) - s({\bf x}_k)\right)^2$, and the empirical RASE by
\begin{eqnarray}\label{RASE}
RASE_s=\frac{1}{M}\sum_{z=1}^M RASE^{[z]}_{s}. 
\end{eqnarray}
Let us also define the empirical  squared deviation at point  ${\bf x}_k$ by  $\nu_k = \frac{1}{M}\sum_{z = 1}^{M}R_s^{[z]}(k)$, and empirical variance of the squared deviation at ${\bf x}_k$ by  $\sigma^2_s(k) = \frac{1}{M - 1} \sum_{z= 1}^{M}\left(R_s^{[z]}(k) - \nu_k\right)^2$. From these quantities we deduce the averaged variance of the squared deviations defined by 
\begin{eqnarray}\label{sigma_s}
\sigma^2_s = \frac{1}{K} \sum_{k = 1}^{K}\sigma^2_s(k).
\end{eqnarray}
In all the simulation setups, we use  the same   mixing proportion function $\pi(\cdot)$: 
\begin{eqnarray*}
\pi({\bf x}) &=& \frac{\operatorname{sin}(3  \pi  x) - 1}{15} + 0.4,  \quad {\bf x}\in [0,1]. \\
\end{eqnarray*}
{\bf Gaussian  setup {\bf (G)}}. The errors $\varepsilon_{j,i}({\bf x})$'s  are distributed  according to a Gaussian topographical scaling model corresponding to  (\ref{topovar}), i.e.  $f$ is the ${\mathcal N}(0,1)$ p.d.f. when  the location and scaling functions  are
\begin{eqnarray*}
a({\bf x}) = 4 - 2  \operatorname{sin}(2\pi {\bf x}),~~ b({\bf x}) = 1.5 \operatorname{cos}(3\pi {\bf x})  - 3,~~ \sigma({\bf x}) = 0.9 \operatorname{exp}({\bf x}), \quad {\bf x}\in [0,1].
\end{eqnarray*}
{\bf Student   setup {\bf (T)}}. The errors  $\varepsilon_{j,i}({\bf x})$'s  are distributed  according to a Student distribution with continuous  degrees of freedom function  denoted $df({\bf x})$. 
The  locations  and degrees of freedom  functions  are  
\begin{eqnarray*}
a({\bf x})= 3 - 2  \operatorname{sin}(2\pi x), ~~b({\bf x}) = 1.5 \operatorname{cos}(3\pi x)  - 2, ~~df({\bf x}) = -5  x + 8, \quad {\bf x}\in [0,1].
\end{eqnarray*}
{\bf Laplace   setup {\bf (L)}}.  The errors  $\varepsilon_{j,i}({\bf x})$'s  are distributed  according to a  Laplace distribution with scaling function $\nu({\bf x})$. The locations and scaling functions  are
\begin{eqnarray*}
a({\bf x}) = 5 - 3  \operatorname{sin}(2\pi {\bf x}), ~~b({\bf x}) = 2 \operatorname{cos}(3\pi {\bf x})  - 4,~~ \nu({\bf x}) = {\bf x} + 1, \quad {\bf x}\in [0,1].
\end{eqnarray*}

\noindent {\it Comments on Tables 1-3.} We report for the simulation setups  {\bf (G)}, {\bf (T}) and {\bf (L)}  the quantities $RASE_s$ defined in (\ref{RASE}),  and between parenthesis $\sigma^2_s$ defined in (\ref{sigma_s}),  for $s=\pi, ~a,~b$. In these tables, we label our method as NMR-SE (Nonparametric  Mixture of Regression with Symmetric Errors). To illustrate the contribution of our method, we compare our results with the RASE obtained by using the local EM-type algorithm proposed by Huang et al. (2013) for Nonparametric Mixture of Regression models with Gaussian noises (method labeled for simplicity NMRG).  
When the errors of the simulated model are Gaussian, the NMRG
 estimation should outperform our method, since the NMRG
 method assumes correctly that the errors are normally distributed, while our method does not  make any  parametric assumption on the distribution of the errors. When the sample size $n= 400$, the NMRG
 is more precise than our method, since the $RASE_s$'s and $\sigma_s^2$'s  are both smaller for the NMRG
. When we increase the sample size of the simulated datasets to $n= 800,~1200$, our method becomes more competitive and yields  $RASE_s$'s and $\sigma_s^2$'s that are  lower than those obtained by NMRG
.  This  surprising  behavior  is  probably due to the fact  that  in model (\ref{modelrv1})  we impose 
the equality in law of the noises up to a shift parameter, when in  the NMRG
 approach possibly different variances are fitted to each kind of noise, increasing by the way drastically  the degrees of freedom of the model to be addressed. 
\begin{table}[h!]
\begin{center}
\begin{tabular}{  l l l l l }
Sample size & Method &$RASE_{\pi}~(\sigma^2_\pi)$ & $RASE_{a}~(\sigma^2_a)$ & $RASE_{b}~(\sigma^2_b)$ \\ \hline
\multirow{2}{*}{$n = 400$} & NMRG
 &  0.011 (0.015) & 0.523 (0.952) & 0.237 (0.415) \\
 & NMR-SE
 & 0.018 (0.034) & 0.661 (1.485) & 0.304 (0.833)\\ \hline
\multirow{2}{*}{$n = 800$} & NMRG
 & 0.010 (0.012) & 0.436 (0.767)& 0.206 (0.368) \\
 & NMR-SE
 & 0.006 (0.013) & 0.311 (0.696) & 0.145 (0.370)\\ \hline
\multirow{2}{*}{$n = 1200$} & NMRG
 &0.009 (0.013) & 0.469 (0.896) & 0.197 (0.340) \\
 & NMR-SE
 &  0.003 (0.008) & 0.209 (0.439) & 0.094 (0.230) \\ \hline
\end{tabular}
\caption{$RASE_z$'s and $\sigma^2_z$'s  for data with Gaussian Errors}
\end{center}
\end{table}
In Tables 2 and 3 we observe that  our  method has globally smaller $RASE_s$'s and $\sigma_s^2$'s.  This result is not surprising, given that in the estimation methodology of Huang et al. (2013), the distribution of the noise are then completely misspecified under the simulation setups {\bf (T)} and {\bf (L)}.
Note however, that when the sample size is small $n = 400$, the NMRG
 displays better results, which  can be explained by the fact that when we generate small size datasets, the points that are supposed to be in the tails of the non-normal distributions are less likely to appear in the dataset. So in that case it can be reasonable to assume that the Gaussian distribution approximates the errors distribution well.

\begin{table}[h]
\begin{center}
\begin{tabular}{  l l l l l }
Sample size & Method &$RASE_{\pi}~(\sigma^2_\pi)$ & $RASE_{a}~(\sigma^2_a)$ & $RASE_{b}~(\sigma^2_b)$ \\ \hline
\multirow{2}{*}{$n = 400$} & NMRG
 &  0.013 (0.018) & 0.342 (0.631) & 0.126 (0.205) \\
 & NMR-SE
 & 0.012 (0.025) & 0.294 (0.664) & 0.117 (0.249)  \\ \hline
\multirow{2}{*}{$n = 800$} & NMRG
 & 0.011 (0.014) & 0.236 (0.377) & 0.110 (0.189) \\
 & NMR-SE
 & 0.004 (0.008) & 0.108(0.238) & 0.047 (0.093)  \\ \hline
\multirow{2}{*}{$n = 1200$} & NMRG
 & 0.010 (0.013) & 0.216 (0.352) & 0.099 (0.153) \\
 & NMR-SE
 & 0.003 (0.006) & 0.067 (0.125) & 0.035 (0.072) \\ \hline
\end{tabular}
\caption{$RASE_z$'s and $\sigma^2_z$'s  for data with Student Errors}
\end{center}
\end{table}

\begin{table}[h!]
\begin{center}
\begin{tabular}{  l l l l l }
Sample size & Method &$RASE_{\pi}~(\sigma^2_\pi)$ & $RASE_{a}~(\sigma^2_a)$ & $RASE_{b}~(\sigma^2_b)$ \\ \hline
\multirow{2}{*}{$n = 400$} & NMRG
 &  0.012 (0.004) & 0.250 (0.156) & 0.108 (0.036)\\
 & NMR-SE
 & 0.022 (0.012)  & 0.462 (0.623) & 0.105 (0.088) \\ \hline
\multirow{2}{*}{$n = 800$} & NMRG
 & 0.009 (0.003) & 0.202 (0.100) & 0.091 (0.036) \\
 & NMR-SE
 & 0.004 (0.002) & 0.109 (0.010) & 0.039 (0.014) \\ \hline
\multirow{2}{*}{$n = 1200$} & NMRG
 & 0.009 (0.003) & 0.192 (0.082) & 0.091 (0.035) \\
 & NMR-SE
 & 0.002 (0.001) & 0.064 (0.025) & 0.027 (0.010) \\ \hline
\end{tabular}
\caption{$RASE_z$'s and $\sigma^2_z$'s  for data with Laplace Errors}
\end{center}
\end{table}

\noindent{\it Comments on Figures 1-5}. To illustrate the sensitivity of our method and compare it graphically   to  the NMRG
 approach we plot in  Fig. 1 different  samples  coming  from  the setups {\bf (G)}, {\bf (T)}, and {\bf (L)}  for  $n=1200$, and  in blue lines the corresponding  true location functions $a(\cdot)$ and $b(\cdot)$. In Fig. 2, respectively Fig. 3,  we plot in grey the $M=100$  segment-line interpolation curves  obtained by   connecting the points  $({\bf x}_k,\hat s^{[z]}({\bf x}_k))$, $k=1,\dots,K$ where $s(\cdot)=a(\cdot)$, $b(\cdot)$  for  the  NMRG
 method, respectively our NMR-SE
 method.  In Fig. 4 and 5 we do  the same  for  $s(\cdot)=\pi(\cdot)$.  In  Fig. 2-5  the dashed red lines represent   the mean curves obtained by   connecting the points  $({\bf x}_k, \bar s({\bf x}_k))$, $k=1,\dots,K$ with $\bar s({\bf x}_k)=1/M\sum_{z=1}^M \hat s^{[z]}({\bf x}_k) $ and  $s(\cdot)=a(\cdot)$, $b(\cdot)$ and $\pi(\cdot)$. 
Let us observe first that the good behavior of the NMR-SE
 method  is confirmed by the small variability of the curves in Fig. 3 and 5 compared to those in Fig. 2 and 4  corresponding to the NMRG
 method.   Secondly it is important to notice that sometime, since we did not constrained or method to have $\pi\in [p,P]$ with $0<p<P<1/2$, we  run into some spurious estimation due to label switching or component merging phenomenon.

\noindent {\it Label switching.}  This well known phenomenon,  due to the lack of identifiability when the parametric space is not lexicographically ordered, translate into our case by
a double-representation of the mixture model (\ref{conddensity}), i.e.
\begin{eqnarray*}
 \pi({\bf x })f_{\bf x}(y-a({\bf x }))
+(1-\pi({\bf x}))f_{\bf x}(y-b({\bf x}))= \pi'({\bf x })f_{\bf x}(y-a'({\bf x }))
+(1-\pi'({\bf x}))f_{\bf x}(y-b'({\bf x}))
\end{eqnarray*}
where $a'(\cdot)=b(\cdot)$, $b'(\cdot)=a(\cdot)$, and $\pi'(\cdot)=1-\pi(\cdot)$. This switching phenomenon is observable on the interval $[0,0.2]$ of Fig. 3 (b) where  the two populations of the mixture strongly overlap, see Fig. 1 (b).  \\

\noindent {\it Component merging.} When $\pi(\cdot)$  is close to 0.5  it is  actually hard to decide  if  we have only  one  shifted symmetric distribution, i.e. $g_{\bf x}(y)=1*f(y-c({\bf x}))+0$ where $c({\bf x})=(b+a)({\bf x})/2$ and $f(y)=1/2 f(y+(b-a)({\bf x})/2)+
1/2f(y-(b-a)({\bf x})/2)$ or a balanced two-component mixture $g_{\bf x}(y)=1/2 f(y-a({\bf x}))+1/2 f(y-b({\bf x}))$. This phenomenon happens clearly   when $\hat \pi^{[z]}(\cdot)$ is unexpectedly  attracted  by the single values  0 or 1, as it occurs  sometimes   on  the intervals $[0,0.2]$ or  $[0.8,1]$, see  Fig. 5 (a-c).

\begin{figure}[h!]
\centering
\begin{subfigure}[b]{0.3\textwidth}
\includegraphics[width = \textwidth]{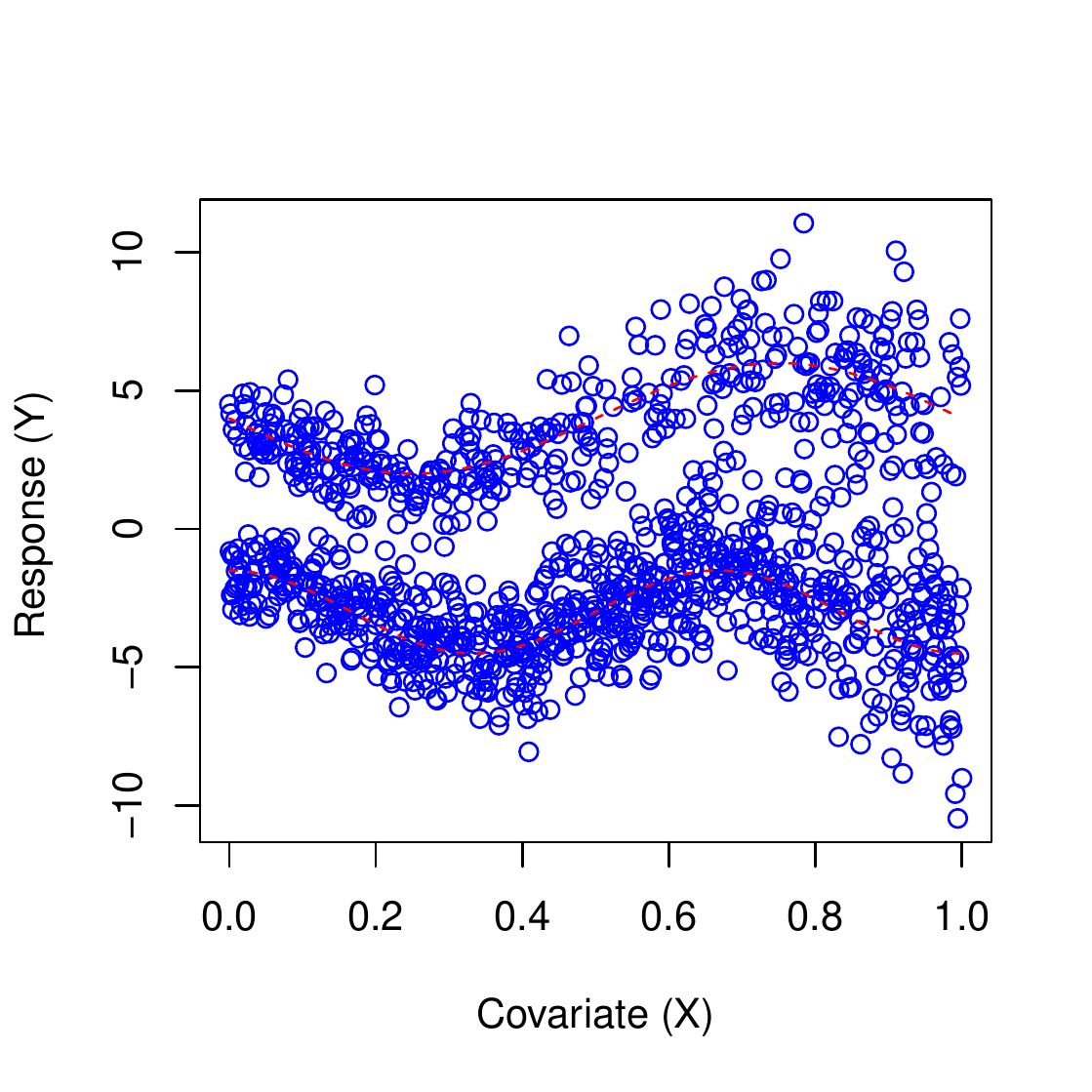}
\caption{Gaussian distribution}
\end{subfigure}
\begin{subfigure}[b]{0.3\textwidth}
\includegraphics[width = \textwidth]{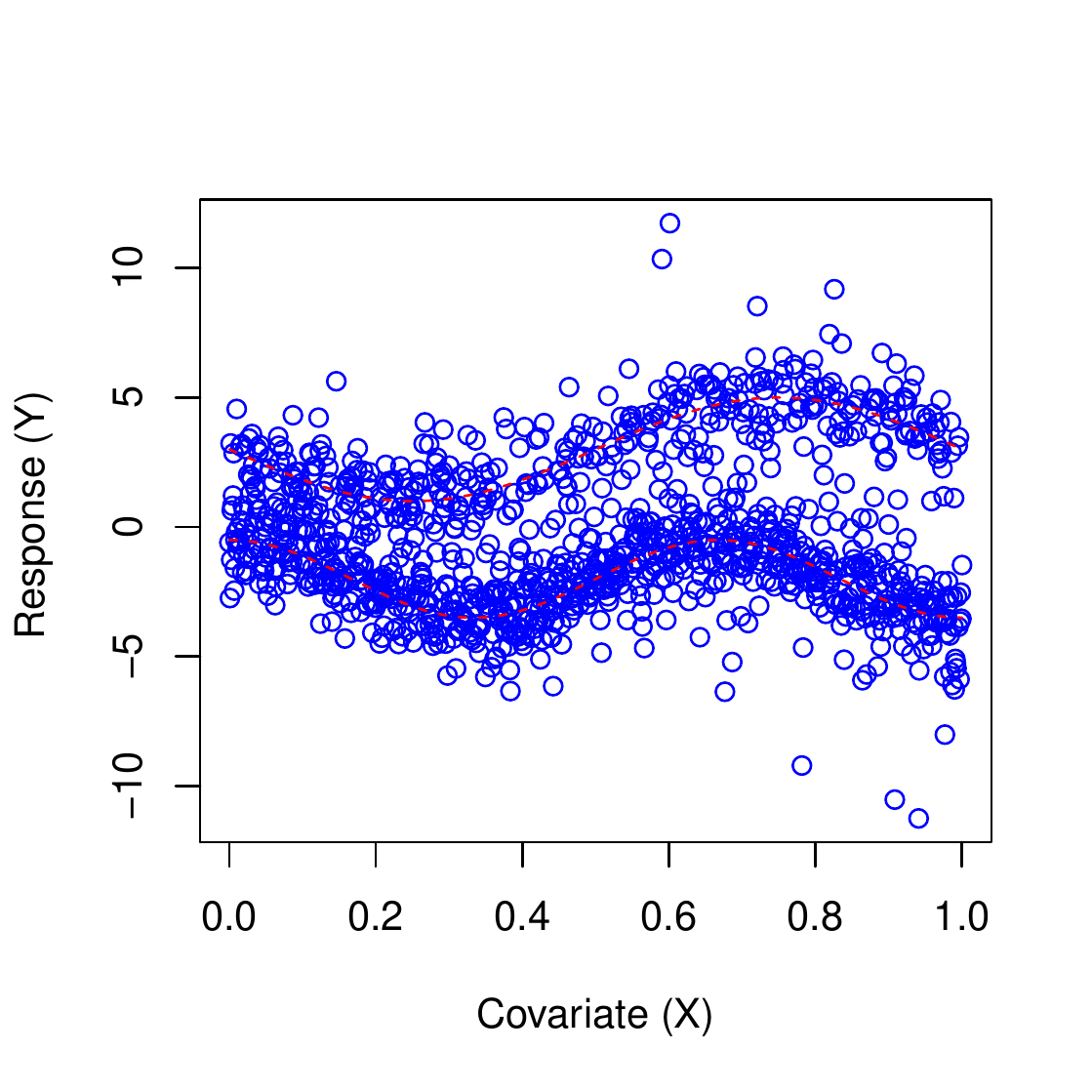}
\caption{Student distribution}
\end{subfigure}
\begin{subfigure}[b]{0.3\textwidth}
\includegraphics[width = \textwidth]{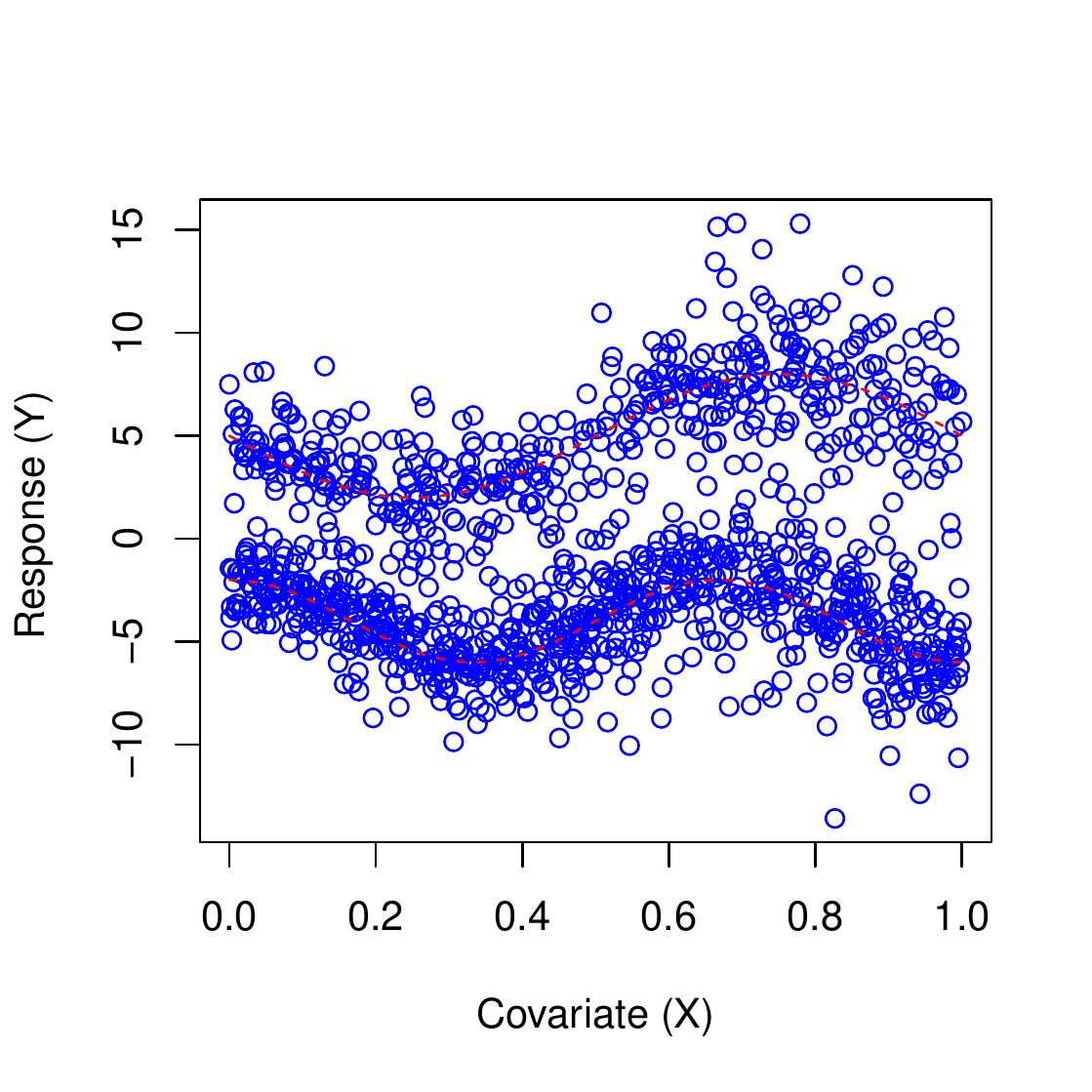}
\caption{Laplace distribution}
\end{subfigure}
\caption{Examples of simulated datasets with different distribution errors}
\end{figure}

\begin{figure}[h!]
\centering
\begin{subfigure}[b]{0.3\textwidth}
\includegraphics[width = \textwidth]{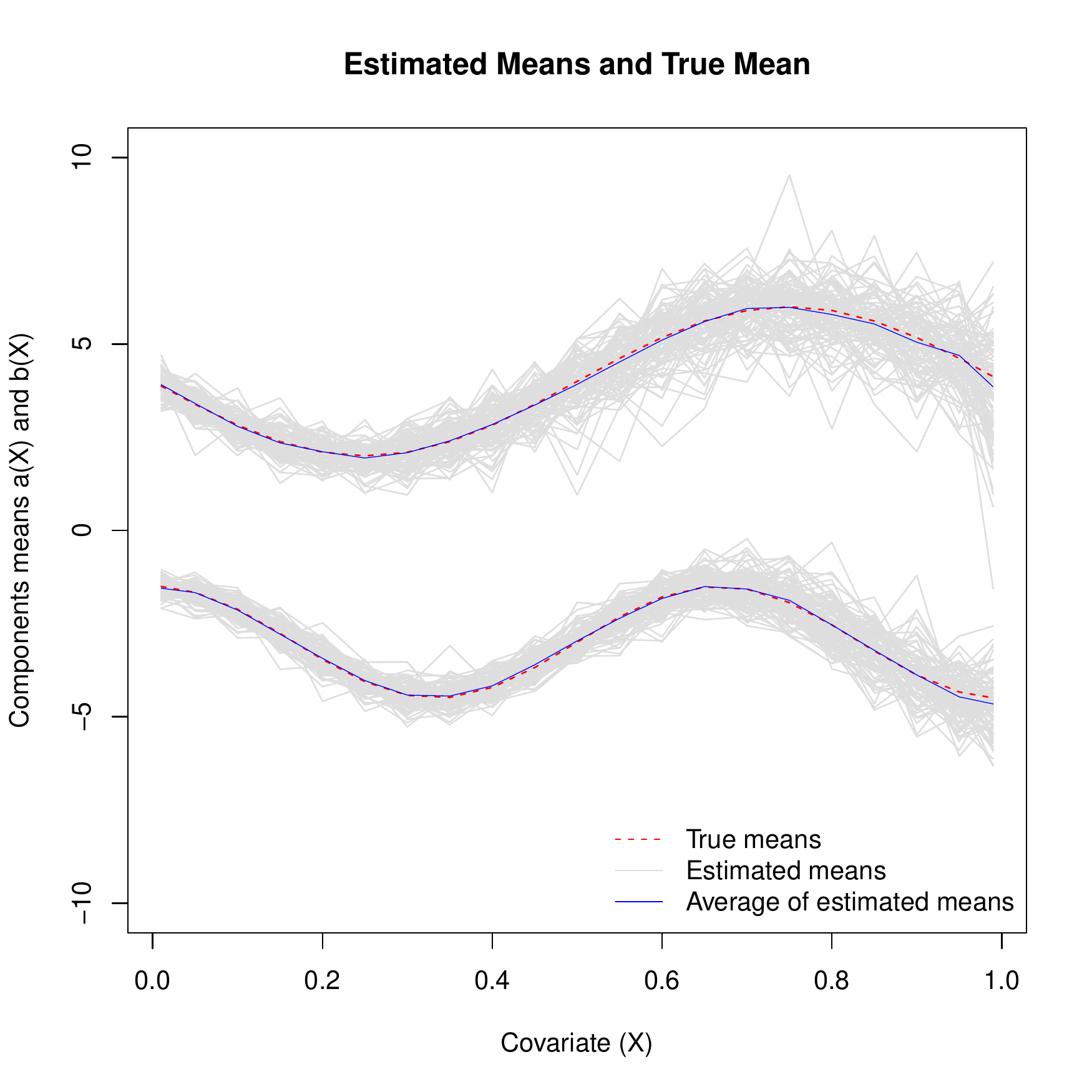}
\caption{Gaussian distribution}
\end{subfigure}
\begin{subfigure}[b]{0.3\textwidth}
\includegraphics[width = \textwidth]{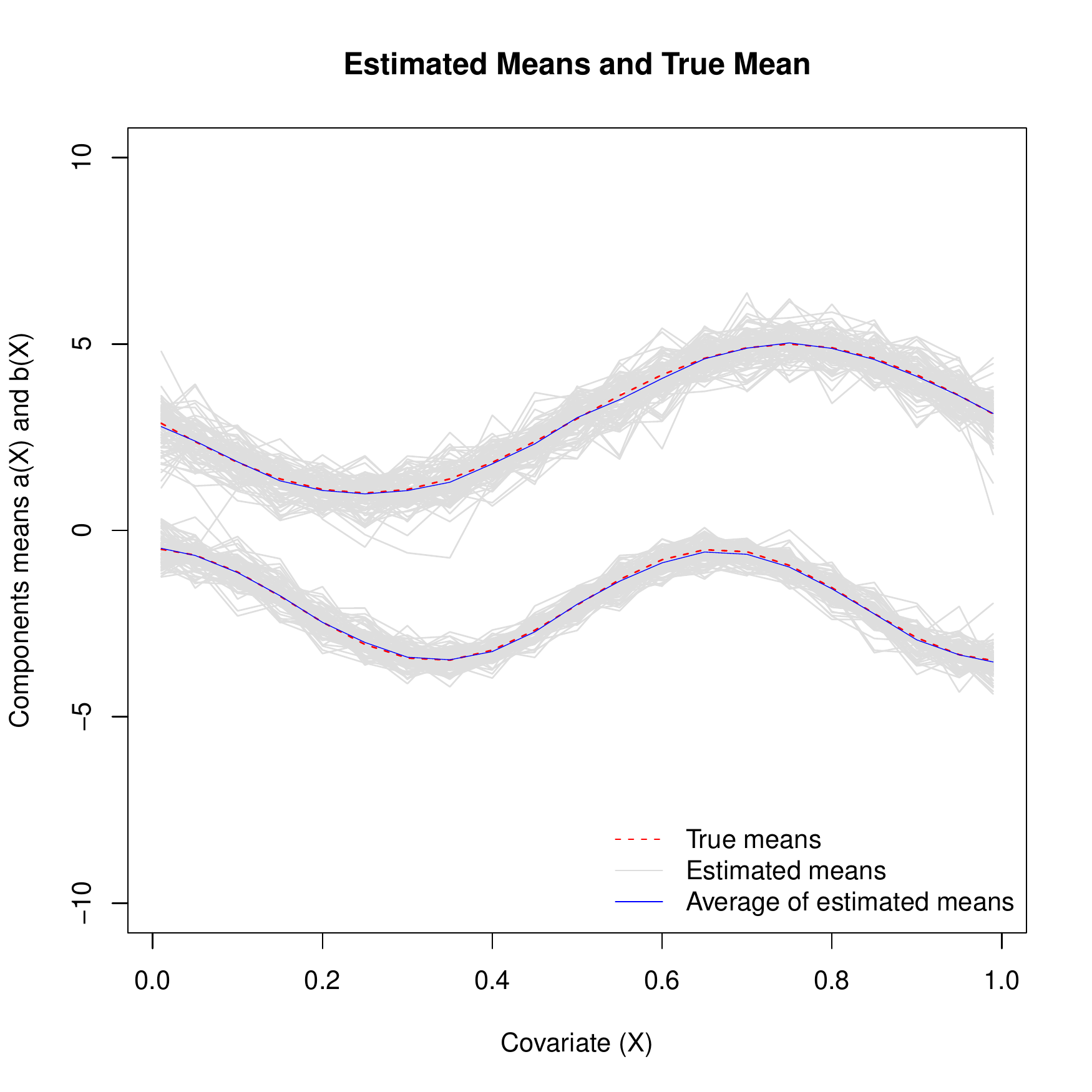}
\caption{Student distribution}
\end{subfigure}
\begin{subfigure}[b]{0.3\textwidth}
\includegraphics[width = \textwidth]{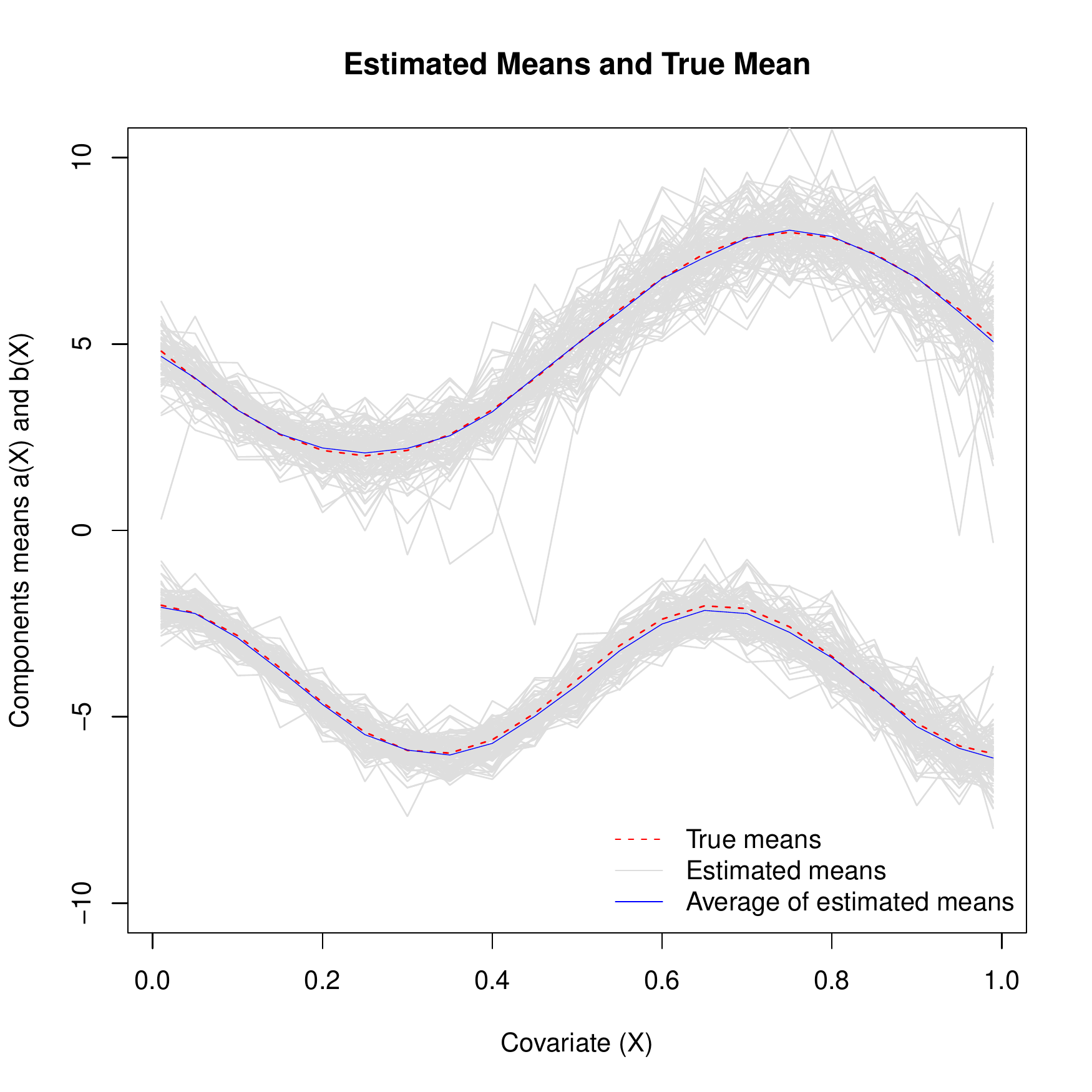}
\caption{Laplace distribution}
\end{subfigure}
\caption{Mean Curves estimated with NMRG
}
\end{figure}

\begin{figure}[h!]
\centering
\begin{subfigure}[b]{0.3\textwidth}
\includegraphics[width = \textwidth]{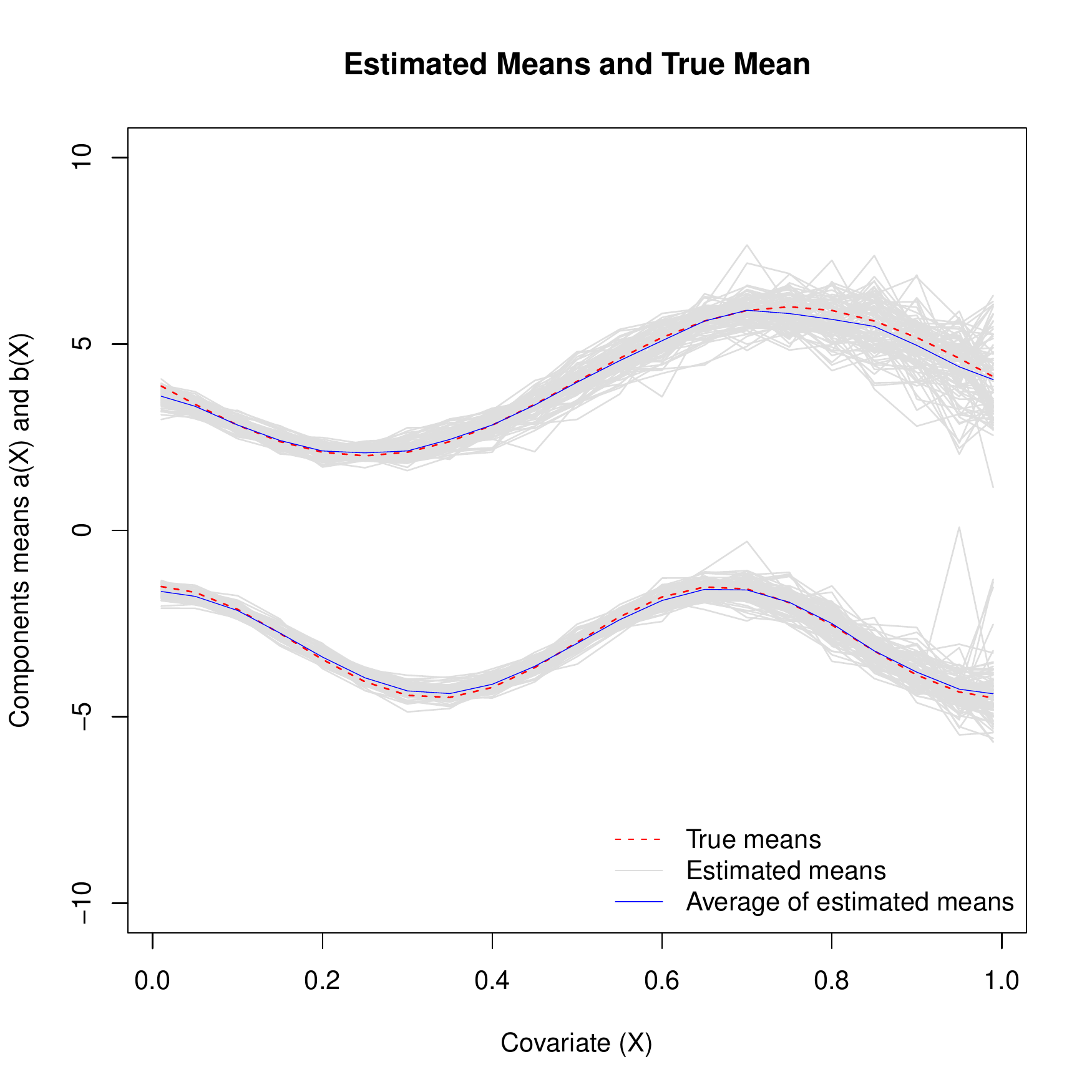}
\caption{Gaussian distribution}
\end{subfigure}
\begin{subfigure}[b]{0.3\textwidth}
\includegraphics[width = \textwidth]{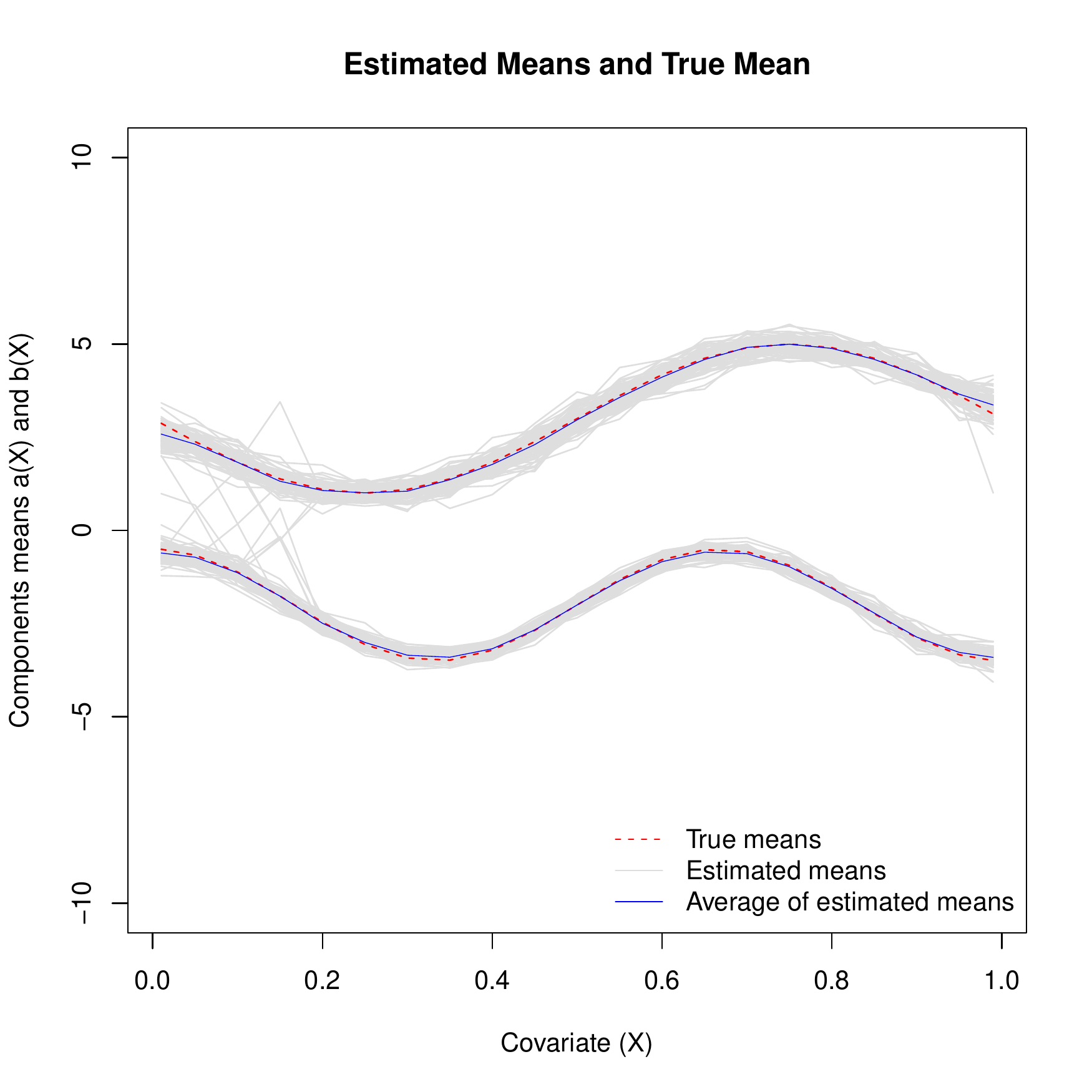}
\caption{Student distribution}
\end{subfigure}
\begin{subfigure}[b]{0.3\textwidth}
\includegraphics[width = \textwidth]{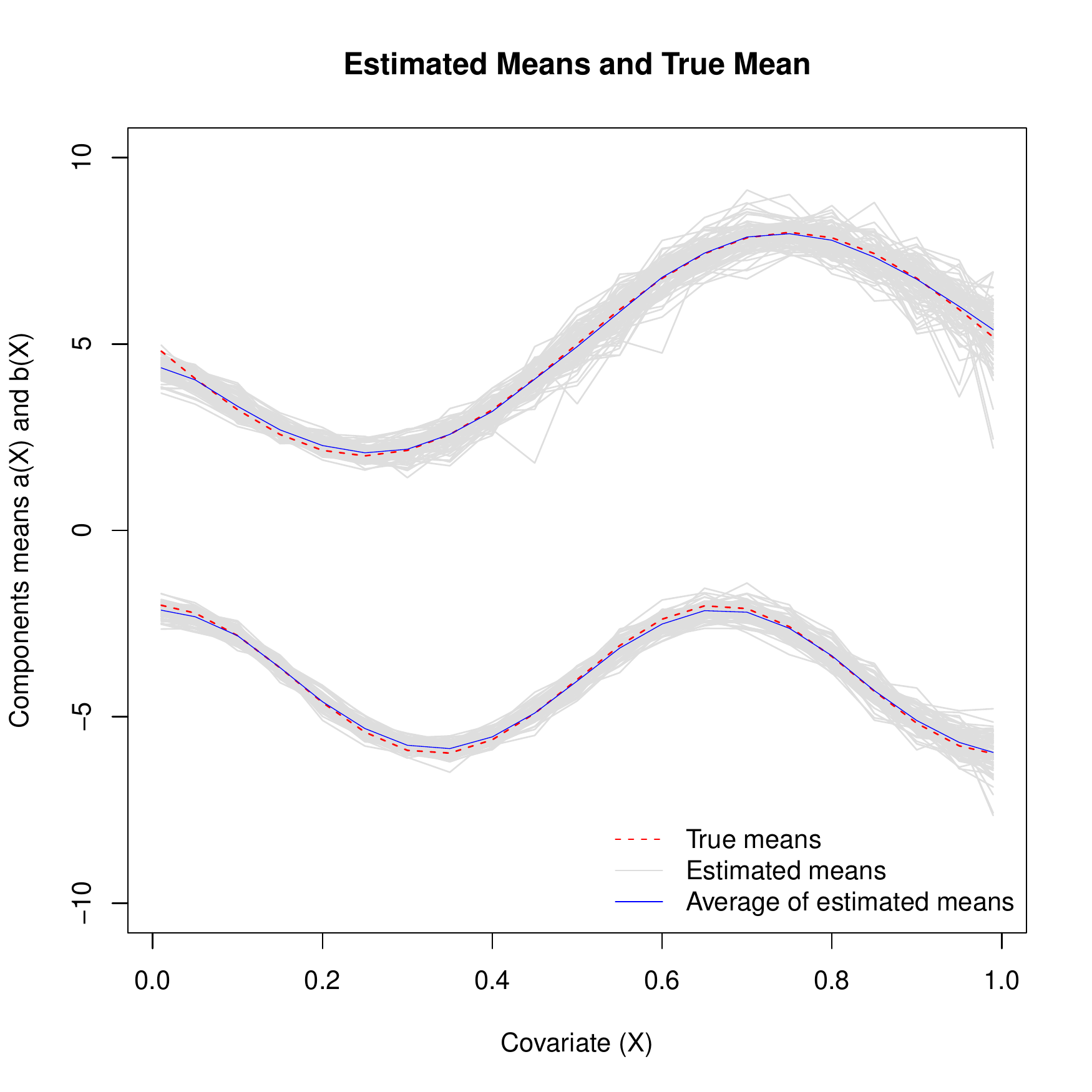}
\caption{Laplace distribution}
\end{subfigure}
\caption{Mean Curves estimated with NMR-SE
}
\end{figure}

\begin{figure}[h!]
\centering
\begin{subfigure}[b]{0.3\textwidth}
\includegraphics[width = \textwidth]{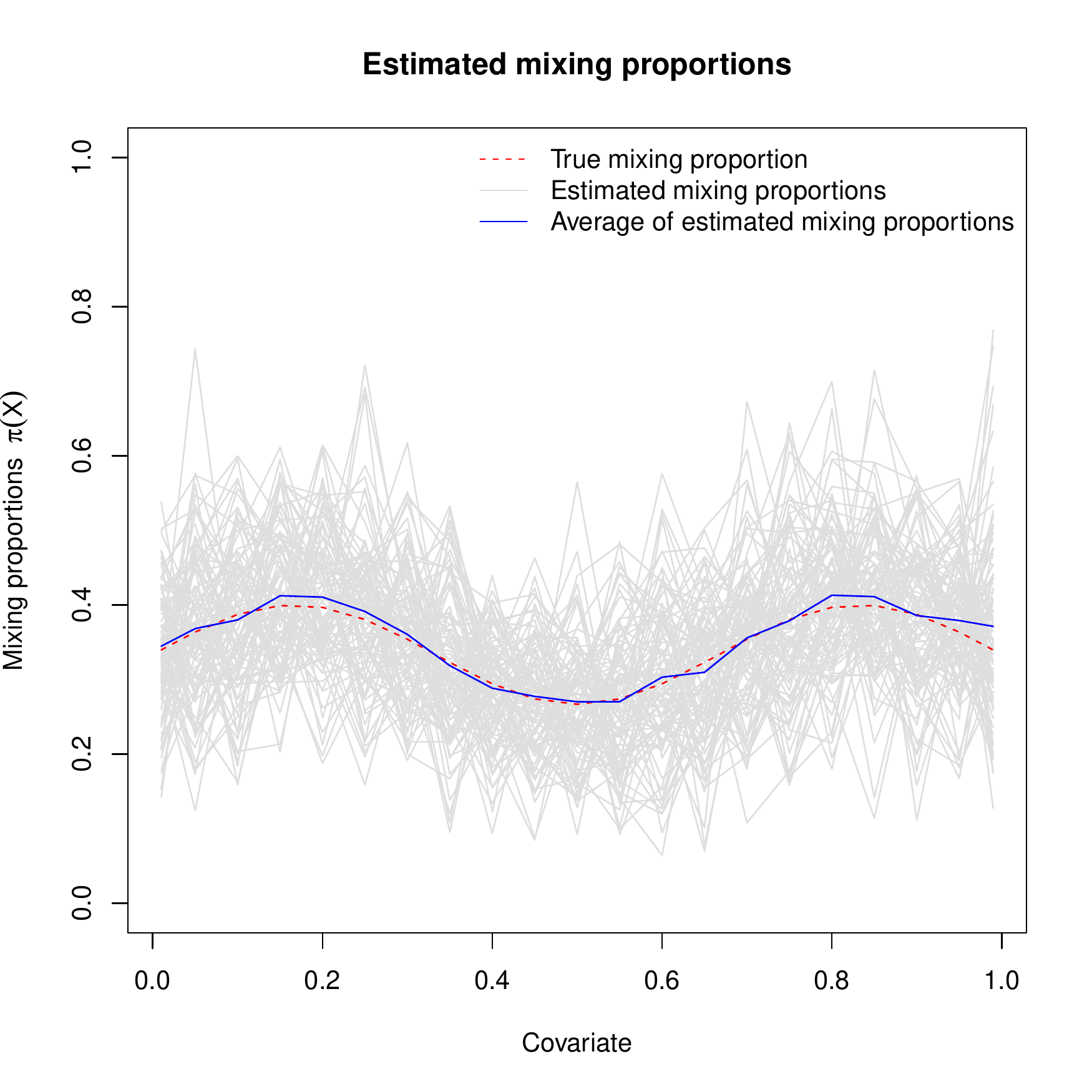}
\caption{Gaussian distribution}
\end{subfigure}
\begin{subfigure}[b]{0.3\textwidth}
\includegraphics[width = \textwidth]{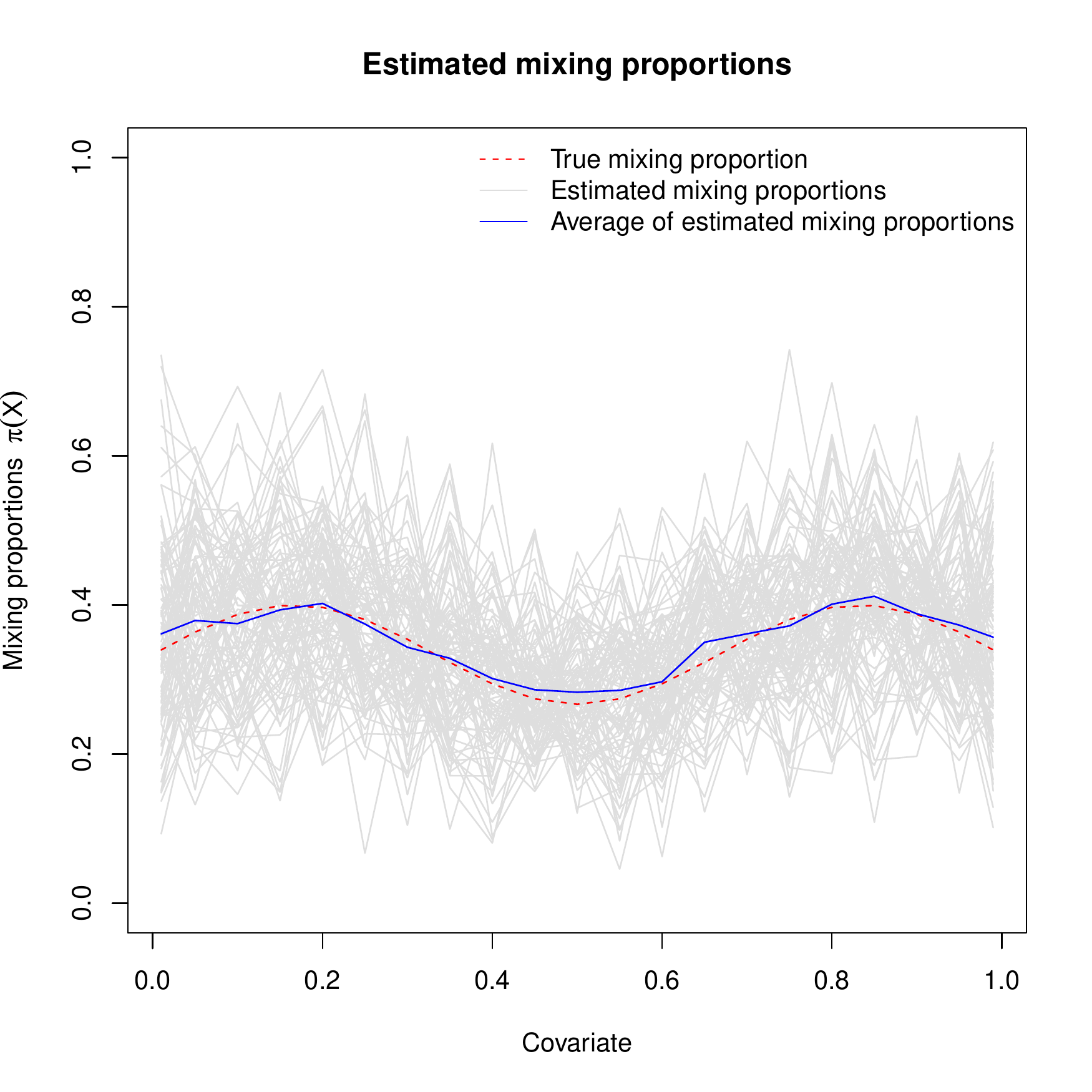}
\caption{Student distribution}
\end{subfigure}
\begin{subfigure}[b]{0.3\textwidth}
\includegraphics[width = \textwidth]{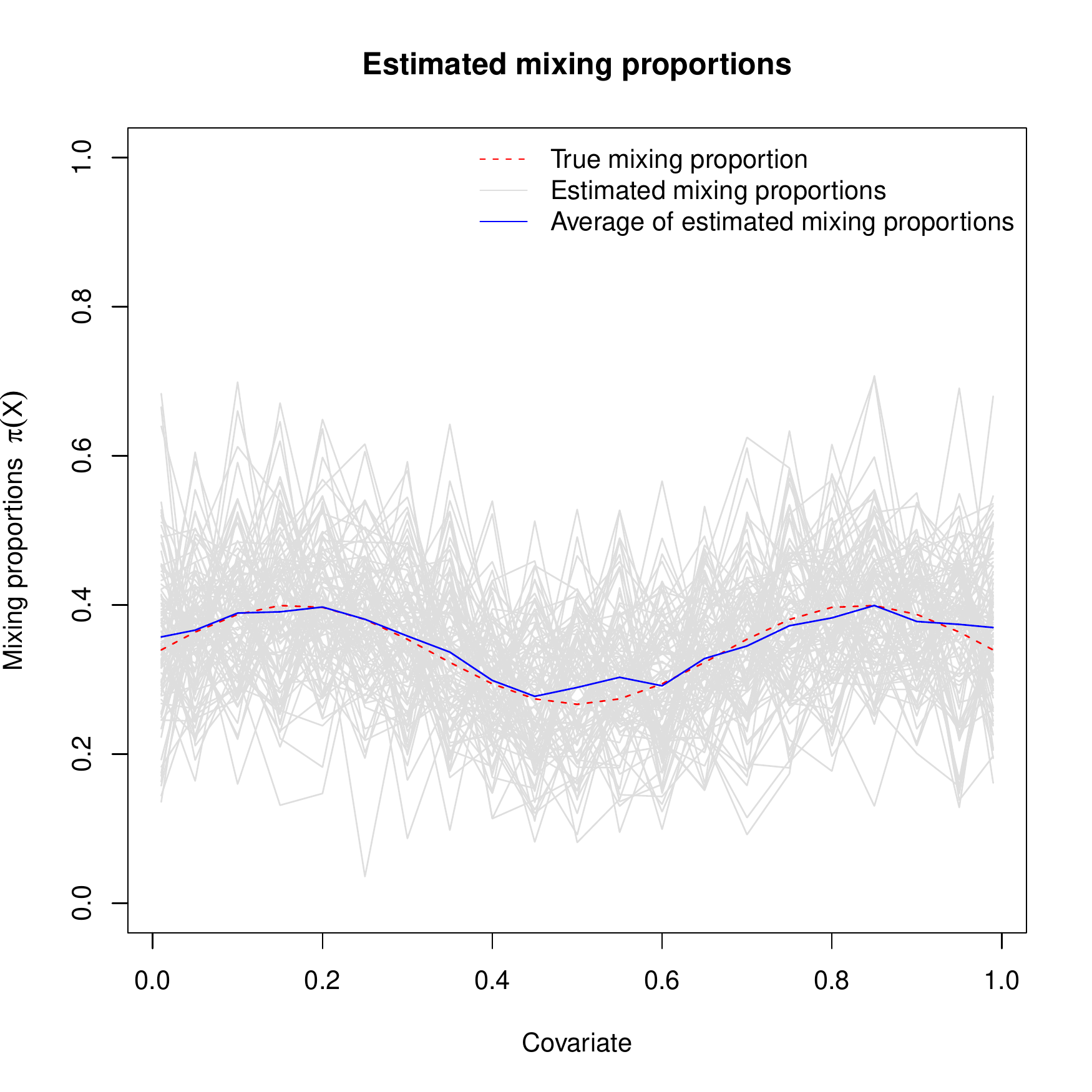}
\caption{Laplace distribution}
\end{subfigure}
\caption{Mixing proportions estimated with NMRG
}
\end{figure}

\begin{figure}[h!]
\centering
\begin{subfigure}[b]{0.3\textwidth}
\includegraphics[width = \textwidth]{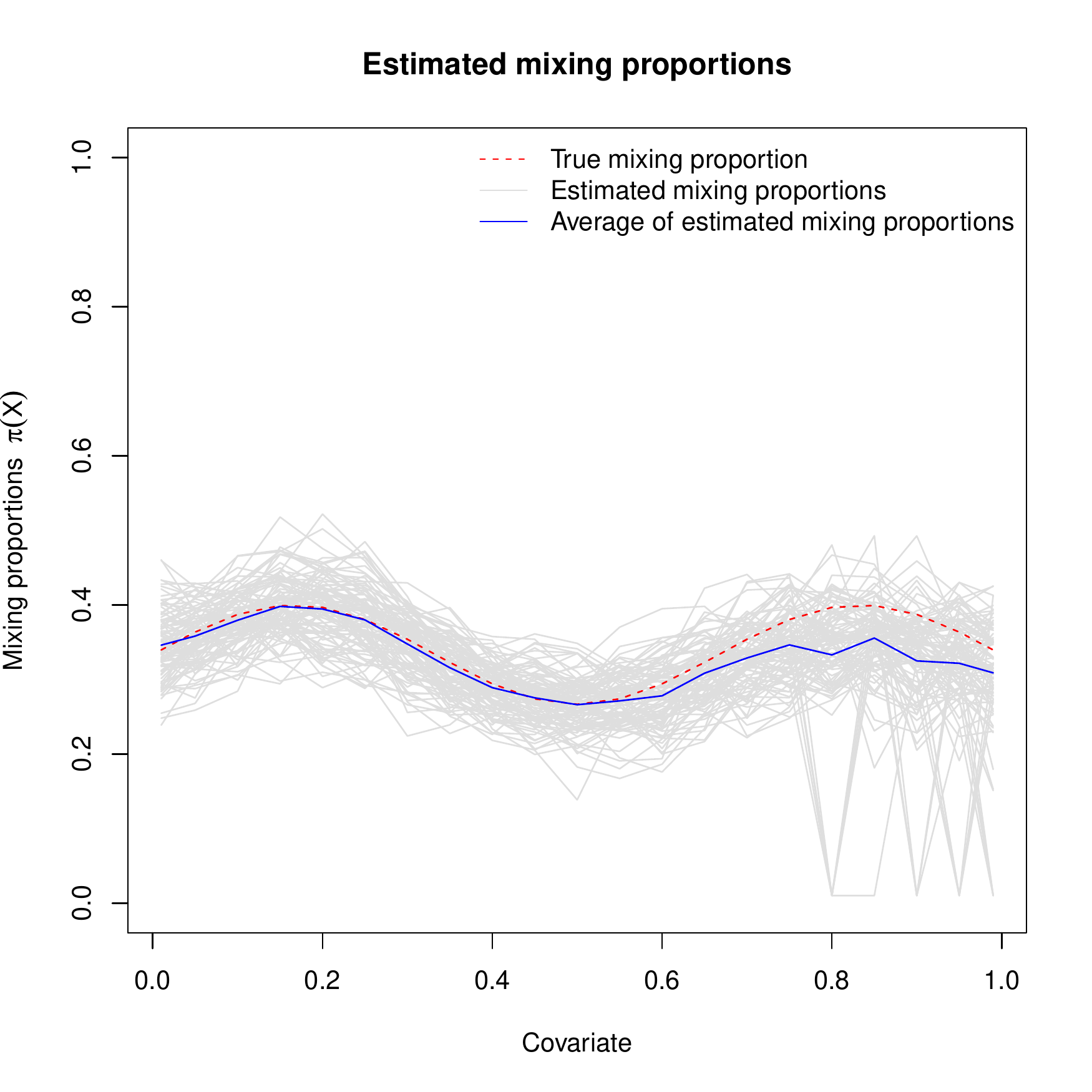}
\caption{Gaussian distribution}
\end{subfigure}
\begin{subfigure}[b]{0.3\textwidth}
\includegraphics[width = \textwidth]{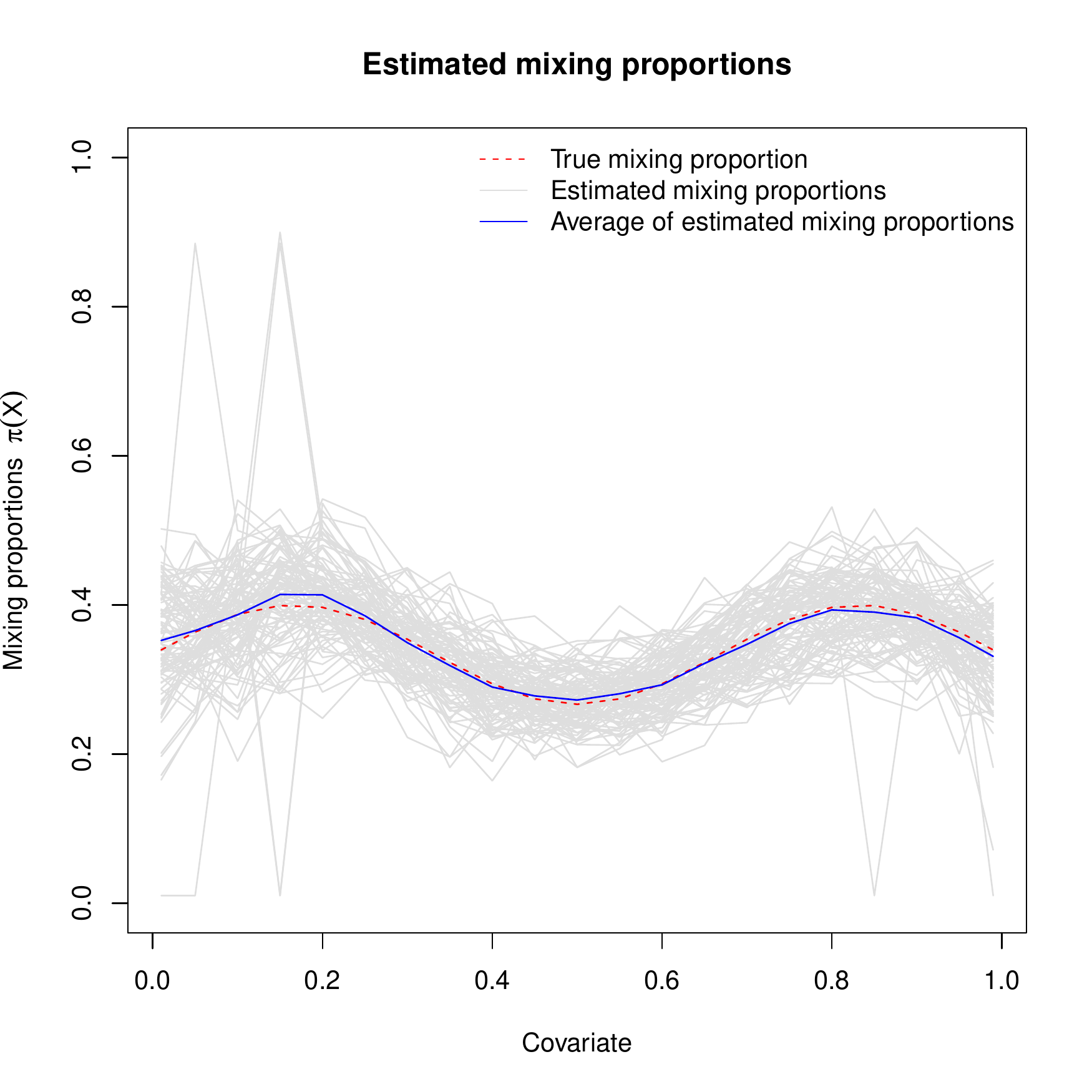}
\caption{Student distribution}
\end{subfigure}
\begin{subfigure}[b]{0.3\textwidth}
\includegraphics[width = \textwidth]{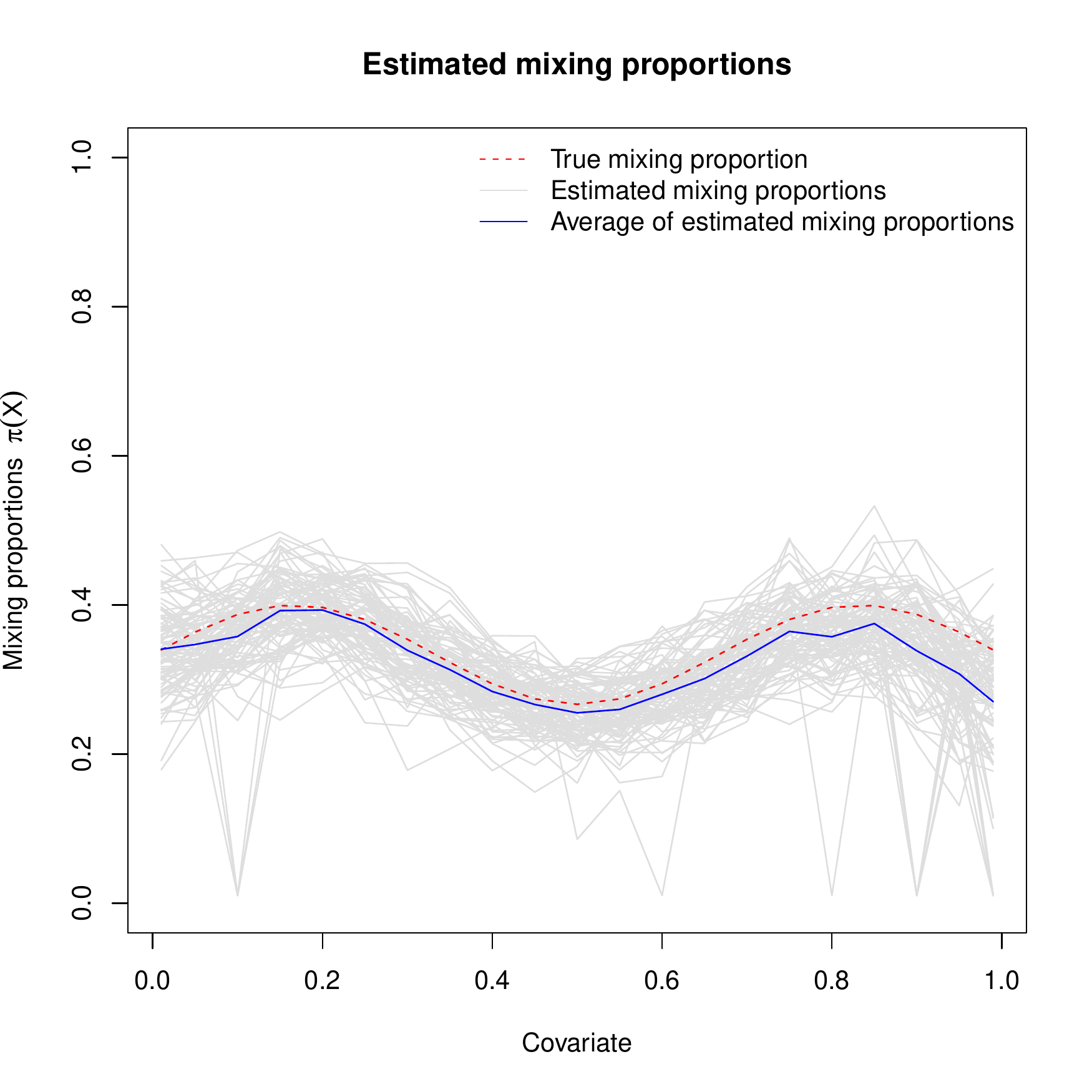}
\caption{Laplace distribution}
\end{subfigure}
\caption{Mixing proportions curves estimated with NMR-SE
}
\end{figure}

\section{Application in radiotherapy}

In this section, we implement the proposed methodology to a dataset obtained from applying radiation therapy to a canine patient with locally advanced Sinonasal Neoplasia. These data were provided by Bowen et al. (2012, Fig. 4) who used them to quantify the associations between pre-radiotherapy and post-radiotherapy PET parameters via spatially resolved mixture of linear regressions.  Intensity Modulated Radiotherapy is an advanced radiotherapy method that uses computer controlled device to deliver radiation of varying intensities to tumor or smaller areas within the tumor. There is evidence showing that the tumor is not homogeneous in its response to the radiation, and that some regions are more resistant than others. Functional imaging techniques (such as Positron Emission Tomography) can be used to identify the radiotherapy resistant regions within the tumor. For instance, an uptake in PET imaging of follow-up 2-deoxy-2-$[^{18}\mbox{F}]$fluoro-D-glucose (FDG)  is empirically linked to a local recurrence of the disease. Bowen et al. (2012), use this approach to construct a prescription function that maps the image intensity values into a local radiation dose that will maximize the probability of a desired clinical outcome. In their manuscript they validate the use of molecular imaging based prescription function against clinical outcome by establishing an association between imaging biomarkers (PET imaging pre-radiotherapy) and regional imaging response to known dosage of therapy (PET imaging post-radiotherapy). The regional imaging response captures the change in imaging signal over an individual image volume element (called a voxel). In our model of interest (\ref{modelrv1}), the pre-radiotherapy PET imaging intensities correspond to the input ${\bf X}_i$'s, and the post-radiotherapy PET imaging levels are the outputs $Y_i$'s.
 For many patients, the empirical link between post-treatment PET of FDG (regional imaging response) and pre-treatment PET of FDG (imaging biomarker at baseline) is well captured by a mixture regression model with two components. For a set of voxels with similar pre-treatment PET intensities, the nature of the response to the radiotherapy leads to two groups of voxels. The first group corresponds to voxels that respond well to the radiotherapy, and the second group contains the non-responding voxels. In our model of interest (\ref{modelrv1}), the non-responding voxel group corresponds to the case where $W({\bf X}_i) = 1$. The location parameters of each group appears to change as the pre-radiotherapy imaging intensity ${\bf X}_i$ varies. These changes in location are captured in our model by the location functions $a(\cdot)$ or $b(\cdot)$, where $a(\cdot)$, respectively $b(\cdot)$, is the component mean function for the completely responding (CR), respectively non-responding (NR), voxel.
 Additionally, the proportion of voxels $\pi({\bf X}_i)$ that respond well to treatment depends on the pre-treatment level of the PET, so the mixture  model should also account for a mixing proportion that depends on the input ${\bf X}_i$.  For a given input ${\bf x}$, we assume that the intensity level of the completely responding and the non-responding voxel have approximately the same p.d.f.  $f_{\bf x}$ up to a shift parameter, with the topographical scaling structure  (\ref{topovar}) presented  in the Introduction.  The variance of the distribution also changes with the level of the covariate (pre-treatment PET FDG). In many cases the variance increases as the intensity of a voxel's PET pre-radiotherapy increases, this is simply due to the fact the responding voxels will have a low post-treatment PET intensity, while the non-responding voxels will not. The aforementioned topographical scaling property, will allow to model this behavior. 
To obtain initial values for the location curves $a(\cdot)$ and $b(\cdot)$, we first use the \textbf{$\bf \texttt{R}$} package \texttt{flexmix}, see  Gruen et. al (2013), which allows us to fit defined parametric functions to the mixture. For the mixing proportion function we set a fixed constant value $\bar{\pi}({\bf x}) = 0.4$. The bandwidths are computed according to the methodology described in Section 4.1, except that the groups are now determined as an output of the \texttt{flexmix} package.  To stress the fact  that the  identification of the topographical model (\ref{modelrv1}) his highly hazardous  in the neighborhood of  the design value 2.5 due to a component crossing (local  non-identifiability),   we plot in dashed line the behavior of our method over the interval $[2,3]$ and will rule out  this domain  from  the following discussion.

In Fig.  6(a), we show the  PET  imaging response to radiotherapy at 3 months, measured by FDG PET uptake, versus the pre-treatment FDG PET uptake and the fitted location functions of the two groups of voxels. For this canine patient, the fitted location curve $a({\bf x})$ of the non-responding voxels increase with the pre-treatment FDG PET uptake, showing a positive relationship between the imaging response and the pre-treatment FDG PET. The location function $b({\bf x})$ corresponding to the completely responding voxels, shows little variation across the range of values of pre-treatment FDG PET and remains relatively flat. This findings are in line with the results obtained by Bowen et al. (2012), however our model is able to capture more than the linear variation in the location curves. Our model also yields the mixing proportions function $\pi({\bf x})$ that can be used to determine the optimal local radiation dose. 
As illustrated in Fig.  6(b),   for this patient voxels
tend to be completely-responding when the pre-treatment FDG PET uptake is between 6.5 and
7.5 SUVs (Standardized Uptake Values), the proportion of non-responding voxels at that
level decreases to 0.25. This suggests that the current radiation dose could be appropriate for voxels that have pre-treatment FDG PET uptake close to the range aforementioned.
In figure 7, we show the estimator $\hat  f_{\bf x}$ of  $f_{\bf x}$, defined in (\ref{local_dens}),   for  different  values of pre-treatment FDG PET uptake ${\bf x}$. 
We see that these conditional distributions are about zero-symmetric  with reasonably small trimming effect due to ${\mathbb I}_{f_n(y|{\bf x}_0
)\geq 0}$ in (\ref{local_dens}) (tiny wave effect on both sides of the main mode). This is a good model validation tool  since  we are actually  able to recover, after  local Fourier inversion,  the basic  symmetry assumption technically made  on the distributions of the errors; see for quality comparison other  existing (nonconditional)  semiparametric inversion density estimates   performed on real datasets: Fig. 1-2 (a) in Bordes et al. (2006), Fig. 3 in Butucea and Vandekerkhove (2013), Fig. 5 in Vandekerkhove (2013), or Fig. 2-3 in Bordes et al. (2013).

\begin{figure}[h!]
\centering
\begin{subfigure}[b]{0.8\textwidth}
\includegraphics[width = \textwidth]{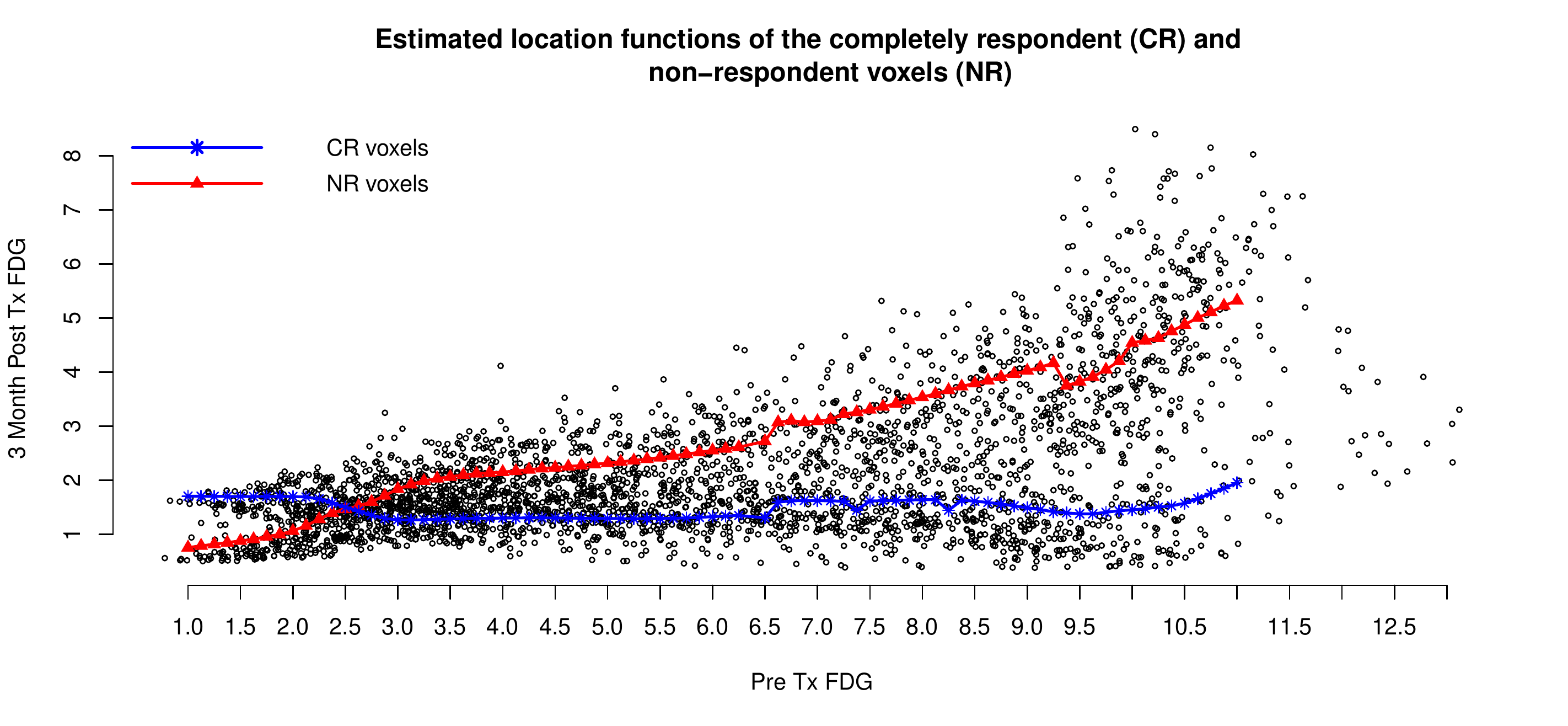}
\caption{\small {Scatter of plots of pre-treatment FDG PET vs. post-treatment FDG PET and estimated location functions for the completely respondent and non-respondent voxel subpopulations}}
\end{subfigure}
\begin{subfigure}[b]{0.8\textwidth}
\includegraphics[width = \textwidth]{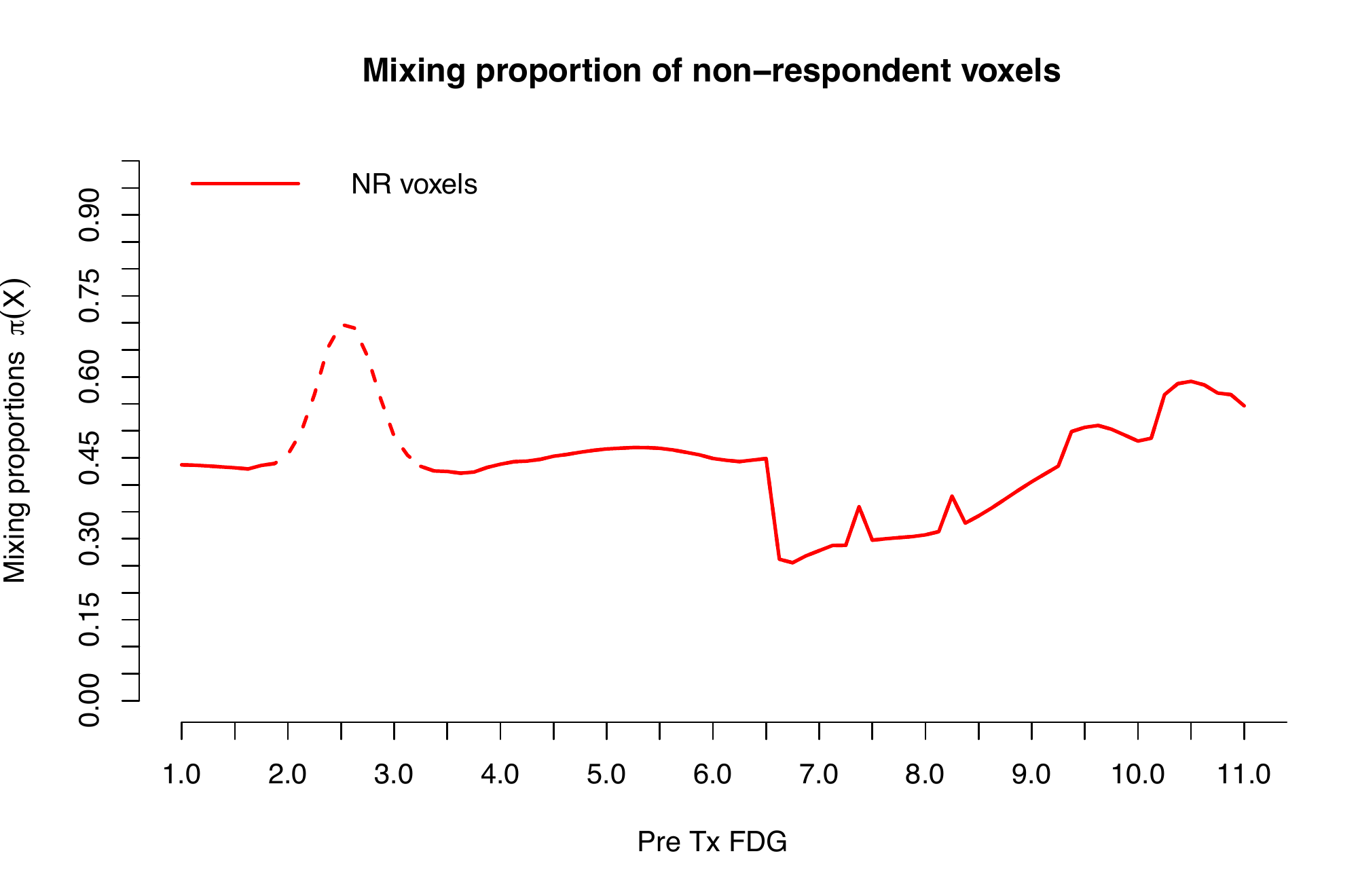}
\caption{Estimated mixing proportions for the completely (CR) and non-respondent (NR) voxel subpopulation}
\end{subfigure}
\caption{}
\end{figure}

\begin{figure}[h!]
\centering
\begin{subfigure}[b]{0.9\textwidth}
\includegraphics[width = \textwidth]{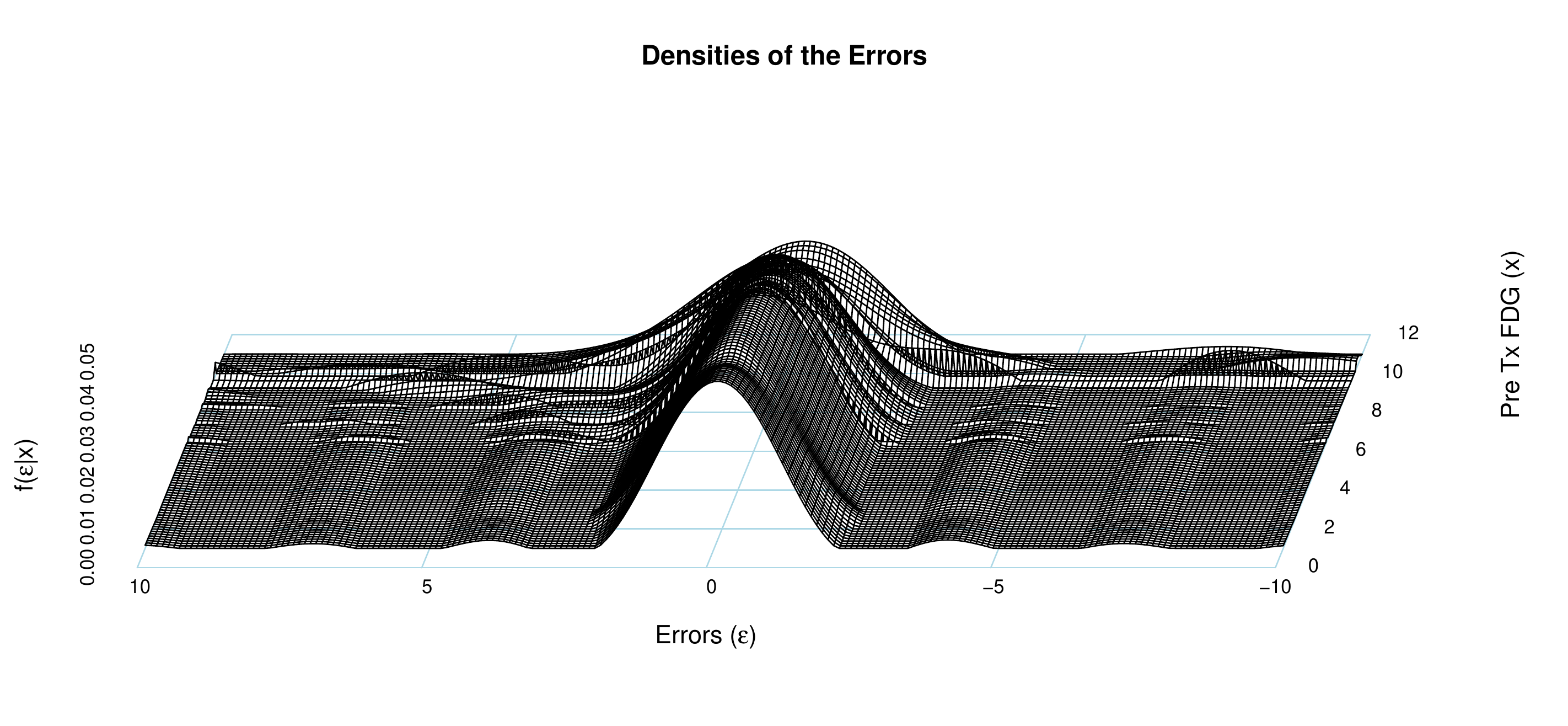}
\end{subfigure}
\caption{Density Estimates of the errors for the different levels of PET Tx FDG values}
\end{figure}

\section{Auxiliary results and main proofs}

Let us denote by $\|\cdot\|$ the Euclidean norm of a vector and by $\|\cdot\|_2$
the Frobenius norm of any squared matrix.
Recall the definition of $Z_k$ in (\ref{eq:Zk}) and let $J(t,u,h):= E[Z_1(t,u,h)]$.
Let $\dot Z_k$ and $\dot J$ denote respectively the gradient of $Z_k$ and $J$ with respect to their first argument $t$.

\begin{lemma}\label{lemma:ZkJ}
Under assumption {\bf A1}  we have:
\begin{enumerate}[i)]
\item  For all $(u, h)\in \R\times \R_+^*$ and any $k=1,...,n$,
$$
\sup_{t\in \Theta} |Z_k(t,u,h)| \leq \frac 2{1-2P} \frac{\|K\|_\infty}{h^d}, \quad
\sup_{t\in \Theta} |J(t,u,h)| \leq \frac 2{1-2P} \|\ell\|_\infty  \cdot\int |K|.$$

\item   For all $(u, h)\in \R\times \R_+^*$ and any $k=1,...,n$,
 \begin{equation*}
\sup_{t\in \Theta} \| \dot Z_k(t,u,h)\| \leq \frac {4(1+|u|)}{(1-2P)^2} \frac{\|K\|_\infty}{h^d}, \quad
\sup_{t\in \Theta} \|\dot J(t,u,h)\| \leq \frac {4(1+|u|)}{(1-2P)^2} \|\ell\|_\infty \cdot  \int |K|.
\end{equation*}

\item  
 For all $(u, h)\in \R\times \R_+^*$ and any $k=1,...,n$,
\begin{eqnarray*}
\sup_{t\in \Theta} \|\ddot Z_k(t,u,h)\|_2 &\leq& \frac {C(1+|u|+u^2)}{(1-2P)^3}\frac{\|K\|_\infty}{h^d},\\
\sup_{t\in \Theta} \|\ddot J_k(t,u,h)\|_2 &\leq& \frac {C(1+|u|+u^2)}{(1-2P)^3}\|\ell\|_\infty \cdot \int |K|,
\end{eqnarray*}
for some   constant $C>0$.

\end{enumerate}
\end{lemma}

\begin{proof}[Proof of Lemma~\ref{lemma:ZkJ}]
i) It is easy to see, from $1-2P\leq |M(t,u)|\leq 1$, that 

 $$|Z_k(t,u,h)| \leq \frac{2}{|M(t,u)|} K_h({\bf X}_k-{\bf x}_0
) \leq \frac{2}{(1-2P)}\frac{\|K\|_\infty}{h^d},$$ and that
$$
|J(t,u)| \leq 2 \left| \int \Im\left( \frac{g_{\bf x}^{\ast }(u)}{M(t,u)} \right) K_h({\bf x}-{\bf x}_0
)\ell({\bf x})d{\bf x})\right|\leq \frac {2 }{(1-2P)}\|\ell\|_\infty. \int |K| .
$$

ii)  We note that
\begin{eqnarray*}
\dot Z_k(t,u,h)
&=& -\left\{\frac{e^{iuY_k}}{M^2(t,u)} \left(\begin{array}{c}
e^{iu\alpha }-e^{iu\beta}\\ iupe^{iu\alpha}\\ iu(1-p) e^{iu \beta}
\end{array} \right)\right.
\\
&&\left.\quad+\frac{e^{-iuY_k}}{M^2(t,-u)} \left( \begin{array}{c}
e^{-iu\alpha }-e^{-iu\beta}\\ -iupe^{-iu\alpha}\\ -iu(1-p) e^{-iu \beta}
\end{array}\right)\right\} K_h({\bf X}_k-{\bf x}_0
),
\end{eqnarray*}
and that
\begin{eqnarray*}
E[\dot Z_k(t,u,h)] = \dot J_k(t,u,h)
&=& -\int \left\{\frac{g_{{\bf x}_0}
(u)}{M^2(t,u)} \left(\begin{array}{c}
e^{iu\alpha }-e^{iu\beta}\\ iupe^{iu\alpha}\\ iu(1-p) e^{iu \beta}
\end{array} \right)\right.\\
&&\left.\quad+\frac{g_{{\bf x}_0}
(-u)}{M^2(t,-u)} \left( \begin{array}{c}
e^{-iu\alpha }-e^{-iu\beta}\\ -iupe^{-iu\alpha}\\ -iu(1-p) e^{-iu \beta}
\end{array}\right)\right\} K_h({\bf x}-{\bf x}_0
)\ell({\bf x})d{\bf x}.
\end{eqnarray*}
We  thus have 

\begin{eqnarray}\label{dot_Z}
\| \dot Z_k(t,u,h)\| &= & \left\| \frac{e^{iuY_k}}{M^2(t,u)} \dot M(t,u)+ \frac{e^{-iuY_k} }{M^2(t,-u)} \dot M(t,-u)\right\|K_h({\bf X}_k-{\bf x}_0
)\nonumber\\
& \leq & \frac 1{(1-2P)^2} \left(2 \left( 2^2+ p^2u^2 +(1-p)^2u^2 \right)\right)^{1/2}K_h({\bf X}_k-{\bf x}_0
)\nonumber\\
&\leq  &\frac{4(1+|u|)}{(1-2P)^2}\frac{\|K\|_\infty}{h^d}\nonumber,
\end{eqnarray}
and
\begin{eqnarray}\label{dot_J}
\| \dot J_k(t,u,h)\| &= &\int  \left\| \frac{g_{\bf x}^{\ast }(u)}{M^2(t,u)} \dot M(t,u)
+ \frac{g_{\bf x}^{\ast }(-u)}{M^2(t,-u)} \dot M(t,-u)\right \|K_h({\bf X}_k-{\bf x}_0
)\ell({\bf x}) d{\bf x}\nonumber\\
& \leq & \frac 1{(1-2P)^2} \left(2 \left( 2^2+ p^2u^2 +(1-p)^2u^2 \right)\right)^{1/2}\int | K_h({\bf X}_k-{\bf x}_0
)\ell({\bf x})|d{\bf x} \nonumber\\
&\leq  &\frac{4(1+|u|)}{(1-2P)^2}\| \ell \|_\infty.   \int | K|\nonumber.
\end{eqnarray}

iii)  Formula of $\ddot M(t,u)$ being  tedious, we shortly write  that
\begin{eqnarray*}
\ddot Z_k(t, u,h) & = &\left\{ -\frac{e^{iu Y_k}}{M^2(t,u)} \ddot M(t,u)
+\frac{e^{-iu Y_k}}{M^2(t,-u)} \ddot M(t,-u)\right.\\
&& \left.+ 2\frac{e^{iu Y_k}}{M^3(t,u)}  \dot M(t,u)  \dot M(t, u)^\top
-2\frac{e^{-iu Y_k}}{\dot M^3(t,-u)}  \dot M(t,-u) \dot M(t,- u)^\top\right\}K_h({\bf X}_k-{\bf x}_0
),
\end{eqnarray*}
and  deduce our bound from the  above expression using  arguments similar to i) and ii).
\end{proof}

\begin{lemma}\label{lipschitz}
\begin{enumerate}[i)]
\item For all $(t,t')\in \Theta^2$, there exists a constant $C_1>0$  such that
\begin{eqnarray*}\label{c1}
|S_n(t)-S_n(t')|\leq C_1\|t-t'\| \sum_{j\ne k, j,k=1}^n \frac{K_h({\bf X}_k-{\bf x}_0
) K_h({\bf X}_j-{\bf x}_0
)}{n(n-1)}.
\end{eqnarray*}
\item For all $(t,t')\in \Theta^2$,   there exists a constant $C_2>0$  such that
\begin{eqnarray*}\label{c1}
\|\ddot S_n(t)-\ddot S_n(t')\|_2\leq C_2\|t-t'\|  \sum_{j\ne k, j,k=1}^n 
\frac{K_h({\bf X}_k-{\bf x}_0
) K_h({\bf X}_j-{\bf x}_0
)}{n(n-1)}.
\end{eqnarray*}
\item There exists some  constants $C_1, \, C_2>0$ depending on $\Theta,\, \alpha, \, M, K $ such that 
$$
E\left[ \left( \sum_{j\ne k, j,k=1}^n \frac{K_h({\bf X}_k-{\bf x}_0
) K_h({\bf X}_j-{\bf x}_0
)}{n(n-1)}
- \ell^2({\bf x}_0
) \right)^2\right] \leq C_1 h ^{2\alpha} + \frac{C_2}{n h^d},
$$
as $h\to 0$ and $nh^d \to \infty$.
\end{enumerate}
\end{lemma}
\begin{proof}
i) By a first order Taylor expansion we have  
\begin{eqnarray*}
 S_n(t) - S_n(t') 
) K_h({\bf X}_j-{\bf x}_0
)\\
= -\frac 1{2n(n-1)}  
\int
(t-t')^{\top}  \sum_{j \not= k, j,k=1}^n\dot Z_k(t_u,u,h)
Z_j(t_u,u,h)dw(u),
\end{eqnarray*}
where for all $u\in \R$, $t_u$ lies in the line segment with extremities $t$ and $t'$.
Therefore, according to calculations made in the proofs of Lemma \ref{lemma:ZkJ} i) and ii), we obtain
$$
|S_n(t) - S_n(t') |\leq \frac 4{(1-2P)^{3}} \|t-t' \|  \int_{\R} (1+|u|) w(u)du \left| \sum_{j \not= k, j,k=1}^n \frac{K_h({\bf X}_k-{\bf x}_0
) K_h({\bf X}_j-{\bf x}_0
)}{n(n-1)} \right|,
$$
which ends the proof of i) by using assumption {\bf A4}.\\

ii)  Let recall first that
$$
\ddot S_n(t) = \frac{-1}{2n(n-1)} \sum_{k\ne j} \int
\left[\ddot Z_k(t,u,h) Z_j(t,u,h) + \dot Z_k(t,u,h) \dot Z_j(t,u)^\top \right]
dw(u).
$$
We shall bound from above as follows
\begin{eqnarray*}
\| \ddot S_n(t,u)-\ddot S_n(t ',u)\|_2
& \leq & \frac 1{2n(n-1)} \sum_{k \ne j}
\left\{ \left \| \int 
(\ddot Z_k(t,u,h) -\ddot Z_k(t',u,h)) Z_j(t,u) dw(u)\right \|_2 \right. \\
&& + \left\|\int
\ddot Z_k(t',u,h) (Z_j(t,u,h) - Z_j(t',u,h)) dw(u) \right\|_2 \\
&& + \left\| \int 
\dot Z_k(t,u,h) (\dot Z_j(t,u,h)- \dot Z_j(t',u,h))^\top dw(u)\right\|_2 \\
&& \left. + \left\| \int 
(\dot Z_k(t,u,h) - \dot Z_k(t ',u,h)) \dot Z_j(t',u,h)^\top dw(u)\right\|_2\right\}.
\end{eqnarray*}
For each term in the previous sum, we use Taylor expansion and upper-bounds similar to those developed in the proof of Lemma \ref{lemma:ZkJ}, and  get
\begin{eqnarray*}
&&\left\| \ddot S_n(t,u)-\ddot S_n(t ',u)\right\|_2\\
&\leq & \left\|t - t' \right\| \frac{C \int (1+|u|+u^2+|u|^3) dw(u)}{(1-2P)^5}
\left| \sum_{j \not= k, j,k=1}^n \frac{K_h({\bf X}_k-{\bf x}_0
) K_h({\bf X}_j-{\bf x}_0
)}{n(n-1)} \right|,
\end{eqnarray*}
for some constant $C>0$, which finishes the proof by using assumption {\bf A4}.

iii) The proof is a consequence of Proposition~\ref{mse_contrast} hereafter.\end{proof}


\begin{proof}[Proof of Proposition~\ref{mse_contrast}]
We shall bound from above the mean square error by the usual decomposition into squared bias plus variance.

Note that
\begin{eqnarray*}
E[S_n(t)] &=& -\frac 14 \int (E[Z_1 (t,u,h)])^2 w(u) du
\end{eqnarray*}
as $(Y_i, {\bf X}_i)$, $i=1,...,n$ are independent. Moreover,
\begin{eqnarray*}
E[Z_1 (t,u,h)]
&=& \int \int \left( \frac{e^{iuy}}{M(t,u)} - \frac{e^{-iuy}}{M(t,-u)}\right) K_h ({\bf x} - {\bf x}_0
) g(y,{\bf x}) dy d{\bf x}\\
&=& \int \left( \int \left( \frac{e^{iuy}}{M(t,u)} - \frac{e^{-iuy}}{M(t,-u)}\right) g_{\bf x}(y) dy \right) \ell({\bf x})  K_h ({\bf x} - {\bf x}_0
) d{\bf x}\\
&=& \int \left( \frac{g^*_{\bf x}(u)}{M(t,u)} - \frac{g^*_{\bf x}(-u)}{M(t,-u)}\right) \ell({\bf x})  K_h ({\bf x} - {\bf x}_0
) d{\bf x}.
\end{eqnarray*}
Let us denote by $ L({\bf x},t,u) := \frac{g^*_{\bf x}(u)}{M(t,u)} - \frac{g^*_{\bf x}(-u)}{M(t,-u)}$, which is further equal to
$$
L({\bf x},t,u) = 2 i \cdot \Im
 \left( \frac{g^*_{\bf x}(u)}{M(t,u)} \right)
= 2 i \cdot \Im
 \left( \frac{M(\theta({\bf x}),u)}{M(t,u)} \right) f^*_{\bf x}(u).
$$
We can write $E[Z_1 (t,u,h)] = [(L(\cdot,t,u)\ell)\star K_h]({\bf x}_0
)$, where $\star$ denotes the convolution product. The bias of $S_n(t)$ is bounded from above as follows:
\begin{eqnarray*}
|E[S_n(t)] - S(t)| &=& \frac 14 \left| \int \left(
[(L(\cdot,t,u)\ell) \star K_h]^2 ({\bf x}_0
)
- L^2({\bf x}_0
,t,u) \ell^2({\bf x}_0
)
\right)  w(u) du \right| \\
&\leq & \frac 14  \int \left|
[(L(\cdot,t,u)\ell) \star K_h ]({\bf x}_0
)
- L({\bf x}_0
,t,u) \ell({\bf x}_0
)\right|\\
&& \cdot \left|
[(L(\cdot,t,u)\ell) \star K_h] ({\bf x}_0
)
+ L({\bf x}_0
,t,u) \ell({\bf x}_0
)\right| w(u) du .
\end{eqnarray*}
Now
$$
|L({\bf x}_0
,t,u) \ell({\bf x}_0
)| \leq \frac{4 \|\ell\|_\infty}{1-2P}
\leq \frac{4 C}{1-2P},
$$
as $\|\ell\|_\infty$ is further bounded by a constant $C=C(\alpha,M)$ depending only on $\alpha , \, M>0$, uniformly over $\ell \in L(\alpha,M)$ (see remark following condition {\bf A1}).
We also have
\begin{eqnarray} \label{expbound}
E[Z_1 (t,u,h)] = |[(L(\cdot,t,u)\ell) \star K_h ]({\bf x}_0
)|
& \leq & \int |L({\bf x},t,u)| l({\bf x}) |K|_h({\bf x}- {\bf x}_0
) d{\bf x}\nonumber \\
& \leq & \frac{4C}{1-2P} \int |K|.
\end{eqnarray}
Moreover, for all $u\in \R$,
\begin{eqnarray*}
&&|[(L(\cdot,t,u)\ell) \star K_h ]({\bf x}_0)- L({\bf x}_0,t,u) \ell({\bf x}_0)| \\
&\leq & \int |L({\bf x}+{\bf x}_0, t, u) \ell({\bf x}+{\bf x}_0) - L({\bf x}_0,t,u) \ell({\bf x}_0)| \cdot |K|_h({\bf x}) d {\bf x}\\
&\leq & c (|u| + \varphi(u)) \int \|{\bf x}\|^\alpha \cdot |K|_h ({\bf x}) d{\bf x}\leq c \cdot h^\alpha (|u|+\varphi(u)) \int \|{\bf x}\|^\alpha \cdot |K| ({\bf x}) d{\bf x} ,
\end{eqnarray*}
under our assumptions {\bf A1-A4}.  Indeed, that implies that $L(\cdot,t,u) \ell(\cdot)$ is Lipschitz $\alpha$-smooth for all $(t,u)\in \Theta\times \R$, with some constant $c>0$, see Lemma~\ref{regLl}. Therefore we get
\begin{equation*}\label{bias}
|E[S_n(t)] - S(t)|
 \leq  \frac{4 C (1+ \int|K|)}{1-2P}\, c \left (\int \|{\bf x}\|^\alpha \cdot |K| ({\bf x} ) d{\bf x}\right ) \cdot \left (\int |u| w(u)du \right ) \cdot h^\alpha.
\end{equation*}
Similarly  to  $S_n(t)$  variance decomposition, we write
\begin{eqnarray*}
&& S_n(t) - E[S_n(t)] \\
&=& \frac{-1}{4n(n-1)} \sum_{j \ne k} \left(
\int (Z_j(t,u,h) Z_k(t,u,h) - E^2[Z_1(t,u,h)]) w(u) du
\right) \\
&=&  \frac{-1}{2n} \sum_j \int ( Z_j(t,u,h)- E[Z_1(t,u,h)]) E[Z_1(t,u,h)]w(u) du\\
&&+\frac{-1}{4n(n-1)} \sum_{j \ne k} \left(
\int (Z_j(t,u,h) - E[Z_1(t,u,h)])(Z_k(t,u,h)- E[Z_1(t,u,h)] ) w(u) du
\right) \\
&=& T_1+T_2, \text{ say}.
\end{eqnarray*}
Terms in $T_1$ and $T_2$ are uncorrelated and thus $Var(S_n(t)) = Var (T_1) + Var(T_2)$.

\noindent On the one hand,
\begin{eqnarray*}
Var(T_1) &= & \frac 1{4 n} Var\left (\int ( Z_1(t,u,h)- E[Z_1(t,u,h)]) E[Z_1(t,u,h)]w(u) du\right )\\
&= & \frac 1{4 n} E \left[ \left(\int ( Z_1(t,u,h)- E[Z_1(t,u,h)]) E[Z_1(t,u,h)]w(u) du\right )^2 \right] \\
&\leq & \frac 1{4 n} E \left[ \int ( Z_1(t,u,h)- E[Z_1(t,u,h)]) ^2 w(u) du  \right] 
\int E^2 [Z_1(t,u,h)]w(u) du,
\end{eqnarray*}
according to  Cauchy-Schwarz inequality.
Now we use (\ref{expbound}) and obtain
\begin{eqnarray*}
Var(T_2) &\leq & \frac 1{4 n} \left( \frac{4C \int |K|}{1-2P}\right)^2
\int E[Z_1(t,u,h)^2] w(u) du.
\end{eqnarray*}
We have,
\begin{eqnarray*}
E[Z_1(t,u,h)^2] & = & E \left[ E\left[\left(2 i\cdot \Im
 \left(\frac{e^{iuY}}{M(t,u)}\right)\right)^2 \middle |  {\bf X}\right] (K_h({\bf X} - {\bf x}_0
))^2
\right] \\
&=& -4 E\left[ \left (\Im
\left( \frac{g^*_{\bf X}(u)}{M(t,u)}\right)\right)^2 (K_h({\bf X} - {\bf x}_0
))^2\right]\\
&\leq & \frac 4{(1-2P)^2} \int \frac 1{h^{2d}} K^2 \left( \frac{{\bf x}-{\bf x}_0
}h \right) \ell({\bf x}) d{\bf x}\\
& \leq & \frac {4 C \int K^2}{(1-2P)^2 h^d}.
\end{eqnarray*}
Therefore,
\begin{eqnarray}\label{varT1}
Var(T_1) &\leq & \frac{16 C^3 (\int |K|)^2 \int K^2}{(1-2P)^4 n h^d} ,
\end{eqnarray}
for all $t \in \Theta$, $h>0$.\\

\noindent On the other  hand,
\begin{eqnarray*}
Var(T_2) &=& \frac{1}{16n(n-1)}E\left[ \left(
\int (Z_1(t,u,h) - E[Z_1(t,u,h)])(Z_2(t,u,h)- E[Z_1(t,u,h)] ) w(u) du\right)^2 \right]\nonumber\\
&\leq& \frac{1}{16n(n-1)}E\left[ 
\int (Z_1(t,u,h) - E[Z_1(t,u,h)])^2(Z_2(t,u,h)- E[Z_1(t,u,h)] )^2 w(u) du\right]\nonumber\\
&\leq& \frac{1}{16n(n-1)}
\int E^2[Z_1(t,u,h)^2] w(u) du\nonumber\\
&\leq&  \frac{1}{16n(n-1)} \left (\frac {4 C \int K^2}{(1-2P)^2 h^d}\right)^2\nonumber\\
&=& \frac{C^2(\int K^2)^2}{n(n-1)(1-2P)^4 h^{2d}},
\end{eqnarray*}
which is clearly a $o((nh^d)^{-1})$ and concludes the proof.
\end{proof}


\begin{lemma}{\bf (Smoothness of $L({\bf x},t,u)\ell({\bf x})$)}\label{regLl}
Assume {\bf A1}-{\bf A4}. There exists a constant $C>0$, such that for all $({\bf x},{\bf x}')\in \R^{d}\times \R^d$ and all $(t,u)\in \Theta\times \R$:
$$
|L({\bf x},t,u)\ell({\bf x})-L({\bf x}',t,u)\ell({\bf x}')|\leq C(|u|+\varphi(u))\|{\bf x}-{\bf x}'\|^{\alpha}.
$$ 
\end{lemma}
\begin{proof}
For $t=(\pi,a,b)\in \Theta$, and $({\bf x},u)\in \R^d\times \R$ we write
\begin{eqnarray*}
L({\bf x},t,u)\ell({\bf x})=f_{{\bf x}_0}
(u)\ell({\bf x}){{\mathcal T}}({\bf x},t,u), ~\mbox{and}~
{\mathcal T}({\bf x},t,u):=\frac{\sum_{i=1}^4{ \mathcal T}_i({\bf x},t,u)}{1-2\pi(1-\pi)\cos[u(a-b)]}
\end{eqnarray*}
where 
\begin{eqnarray*}
{ \mathcal T}_1({\bf x},t,u)&=&\pi({\bf x})\pi\sin[u(a({\bf x})-a)],\quad { \mathcal T}_2({\bf x},t,u)=\pi({\bf x})(1-\pi)\sin[u(a({\bf x})-b)], \\
{ \mathcal T}_3({\bf x},t,u)&=&(1-\pi({\bf x}))\pi\sin[u(b({\bf x})-a)], \quad { \mathcal T}_4({\bf x},t,u)=(1-\pi({\bf x})(1-\pi)\sin[u(b({\bf x})-b)].
\end{eqnarray*}
For all $({\bf x},{\bf x}')\in \R^{d}\times \R^d$ we have
\begin{eqnarray*}
&&|L({\bf x},t,u)\ell({\bf x})-L({\bf x}',t,u)\ell({\bf x}')|\\
&&\leq |f_{{\bf x}_0}
(u)\ell({\bf x})||{\mathcal T}({\bf x},t,u)-{\mathcal T}({\bf x}',t,u)|+|{\mathcal T}({\bf x},t,u)||f_{{\bf x}_0}
(u)\ell({\bf x})-f_{{\bf x}'}^*(u)\ell({\bf x}')|\\
&&\leq \|\ell\|_\infty |{\mathcal T}({\bf x},t,u)-{\mathcal T}({\bf x}',t,u)|+(1-2P)^{-1}|f_{{\bf x}}^*(u)\ell({\bf x})-f_{{\bf x}'}^*(u)\ell({\bf x}')|.
\end{eqnarray*}
Let us now show the $\alpha$-smooth Lipschitz property of  ${\mathcal T}_1$, the proof  for the  other  ${\mathcal T}_i$'s being completely similar. For all  $({\bf x},{\bf x}')\in \R^{d}\times \R^d$
\begin{eqnarray*}
|{\mathcal T}_1({\bf x},t,u)-{\mathcal T}_1({\bf x}',t,u)|&\leq& |\sin[u(a({\bf x})-a)]-\sin[u(a({\bf x}')-a)]|+|\pi({\bf x})-\pi({\bf x}')|\\
&\leq&|u||(a({\bf x})-a({\bf x}')]|+|\pi({\bf x})-\pi({\bf x}')|\\
&\leq&M |u|\|{\bf x}-{\bf x}'\|^\alpha+M\|{\bf x}-{\bf x}'\|^\alpha.
\end{eqnarray*}
On the other hand we have
\begin{eqnarray*}
|f^*(u|{\bf x})\ell({\bf x})-f^*(u|{\bf x}')\ell({\bf x}')|&\leq& |\ell({\bf x})-\ell({\bf x}')|+\|\ell\|_\infty |f_{{\bf x}_0}
(u)-f_{{\bf x}'}^*(u)|,\\
&\leq&(M+\|\ell\|_\infty \varphi(u))\|{\bf x}-{\bf x}'\|^\alpha,
\end{eqnarray*}
which concludes the proof.
\end{proof}


\begin{proof}[Proof of Theorem~\ref{cons}]
Our method is based on a
consistency proof for mininum contrast estimators by
Dacunha-Castelle and Duflo (1993,  pp.94--96). Let us consider a
countable dense set $D$ in $\Theta$, then $ \inf_{t\in \Theta}S_{n}(t)=\inf_{t\in D}S_{n}(t) $, is a
measurable random variable. We define in addition the random
variable
$$
W(n,\xi)=\sup\left\{|S_{n}(t)-S_{n}(t')|;~(t,t')\in D^2,~ \|t-t' \|\leq \xi\right\},
$$
and recall that $S(\theta_0
)=0$. Let us consider a non-empty  open
ball $B_*$ centered on $\theta_0
$ such that $S$ is bounded from
below by a positive real number $2\varepsilon$ on $\Theta\backslash
B_*$. Let us  consider  a sequence  $(\xi_p)_{p\geq 1}$ decreasing to zero,
and take $p$ such that there exists a covering of $\Theta\backslash
B_*$ by a finite number $\kappa$ of balls $(B_i)_{1\leq i\leq \kappa}$
with  centers $ t_i\in \Theta$, $i=1,\dots,\kappa$,   and  radius
less than $\xi_p$. Then, for all $t \in B_i$, we have
\begin{eqnarray*}
S_{n}(t)&\geq& S_{n}(t_i)-|S_{n}(t)-S_{n}(t_i)|\geq S_{n}(t_i)-\sup_{t \in B_i}|S_{n}(t)-S_{n}(t_i)|,
\end{eqnarray*}
which leads to
\begin{eqnarray*}
\inf_{t \in \Theta\setminus B_*}S_{n}(t)\geq
\inf_{1\leq i\leq \kappa}S_{n}(t_i) -W(n,\xi_p).
\end{eqnarray*}
As a  consequence we have the following events inclusions
\begin{eqnarray*}\label{limsup1}
\left\{\hat \theta_n\notin B_*\right\}
& \subseteq &\left\{\inf_{t \in \Theta\setminus B_*}
S_n(t) < \inf_{t \in  B_*}
S_n(t)  < S_{n}(\theta_0
)\right\} \nonumber \\
&\subseteq& \left\{\inf_{1\leq i\leq \kappa} S_{n}(t_i) - W(n,\xi_p)
 < S_{n}(\theta_0
)
\right\} \nonumber \\
& \subseteq &
\left\{W(n,\xi_p)>\varepsilon\right\}\cup  \left\{\inf_{1\leq i\leq \kappa} (S_{n}(t_i)-S_{n}(\theta_0
))\leq \varepsilon\right\}.
\end{eqnarray*}
In addition we have
\begin{eqnarray*}\label{limsup2}
&&P\left( \inf_{1\leq i\leq \kappa} (S_{n}(t_i)-S_{n}(\theta_0
))\leq \varepsilon\right)\\
&& \leq 1-\prod_{i=1}^\kappa (1-[P(|S_{n}(t_i)-S(t_i)|\geq  \varepsilon )+P(|S_{n}(\theta_0
)-S(\theta_0
)|\geq  \varepsilon)]),
\end{eqnarray*}
where, according to Proposition \ref{mse_contrast},  the last two terms in the right hand side 
of the above inequality vanish to zero  if 
$h^{d} n\rightarrow \infty$ and $h\rightarrow 0$ as $n\rightarrow \infty$.
To conclude we use Lemma~\ref{lipschitz} and notice that,   for all $(t,t')\in \Theta^2$,  we have
\begin{eqnarray}\label{c1}
&&|S_n(t)-S_n(t')|\nonumber \\
&&\leq \frac{ C\|t-t'\| }{n(n-1)} \left| \sum_{j\ne k, j,k=1}^n K_h({\bf X}_k-{\bf x}_0
) K_h({\bf X}_j-{\bf x}_0
) \right|\nonumber \\
&&\leq C\|t-t'\|\ell^2({\bf x}_0
)+ C\|t-t'\| \left |\sum_{j\ne k, j,k=1}^n \frac{K_h({\bf X}_k-{\bf x}_0
) K_h({\bf X}_j-{\bf x}_0
)}{n(n-1)}-\ell^2({\bf x}_0
)\right|.
\end{eqnarray}
We deduce from above that
\begin{eqnarray*}
P(W(n,\xi_p)>\varepsilon)
&\leq & P \left ( C\xi_p \ell^2({\bf x}_0
) > \frac{\varepsilon}2\right ) \\
&& +\left(\frac{2 C\xi_p}{\varepsilon}\right)^2 E\left [ \left(\sum_{j\ne k, j,k=1}^n \frac{K_h({\bf X}_k-{\bf x}_0
) K_h({\bf X}_j-{\bf x}_0
)}{n(n-1)}-\ell^2({\bf x}_0
)\right)^2\right],
\end{eqnarray*}
where the last term in the right hand side  is of order  $(nh^{d})^{-1}+h^{2\alpha}$ and tends to 0 by our assumption on $h$.
Since for $p$ sufficiently large we have $C\xi_p \ell^2({\bf x}_0
)<\varepsilon/2$ and thus $P \left (C\xi_p \ell^2({\bf x}_0
) > \varepsilon /2 \right )=0$, this concludes the  proof of the consistency  in probability of $\hat \theta_n$ when $n h^{d}\rightarrow \infty$ and $h\rightarrow 0$ as $n\rightarrow \infty$.
\end{proof}


\begin{proof}[Proof of Theorem~\ref{asymptotic_normality}]
By a Taylor expansion of $\dot{S}_n$ around $\theta_0
$, we have
\begin{eqnarray*}
0 = \dot{S}_{n}(\hat \theta_{n}) = \dot{S}_{n}(\theta_0
) + \ddot S_{n}(\bar \theta_{n})(\hat \theta_{n}-\theta_0
),
\end{eqnarray*}
where $\bar \theta_{n}$ lies in the line segment with extremities $\hat \theta_{n}$ and $\theta_0
$.\\

Let us study the behaviour of
\begin{equation*}
\dot S_n(\theta_0
) = \frac{-1}{2n(n-1)} \sum_{j \ne k}
\int \dot Z_k(\theta_0
,u,h) Z_j(\theta_0
,u,h) w(u) du ,
\end{equation*}
where $\dot Z_k$ denotes the gradient of $Z_k$ with respect to the first argument. Recall that $\theta_0
 = \theta({\bf x}_0
)= (\pi({\bf x}_0
), a({\bf x}_0
), b({\bf x}_0
))$ and therefore
$$
J(t,u,h)=E[Z_1(t,u,h)] = 2i\int \Im
 \left( \frac{M(\theta({\bf x}_0
),u)}{M(t,u)}\right) f_{{\bf x}}^*(u) \ell({\bf x}) K_h ({\bf x} - {\bf x}_0
) d {\bf x} ,
$$
satisfies $J(\theta_0
,u,h)\rightarrow 0$ as $h\to 0$. Indeed, the last integral may be equal to 0 if the set $\{{\bf x}: \theta({\bf x}) = \theta({\bf x}_0
) \}$ has Lebesgue measure 0, or tends (by uniform continuity in ${\bf x}$ of the integrand) to
$$
2i \Im
 \left( \frac{M(\theta({\bf x}_0
),u)}{M(\theta({\bf x}_0
),u)}\right) f_{{\bf x}_0
}^*(u) \ell({\bf x}_0
) = 0.
$$
Moreover,
$$
\dot Z_k(t,u,h) = \Im
 \left( - \dot M (t,u) \frac{e^{iuY_k}}{M^2(t,u)} \right) K_h({\bf X}_k-{\bf x}_0
).
$$
Denote  $\dot J (t,u,h) = E[\dot Z_k(t,u,h)]$ and observe that
$$
\dot J (t,u,h) = \int \Im
 \left( - \dot M (t,u) \frac{M(\theta({\bf x}),u) f^*_{\bf x}(u)}{M^2(t,u)} \right) K_h({\bf x}-{\bf x}_0
) \ell({\bf x}) d{\bf x}.
$$
Then, we decompose $\dot S_n(\theta_0
)$ as follows
\begin{eqnarray}\label{AB}
&& \dot S_n(\theta_0
)\nonumber  \\
&=& \frac{-1}{2n(n-1)} \sum_{j \ne k}
\int \left( \dot Z_k(\theta_0
,u,h) - \dot J(\theta_0
,u,h) \right)
\left( Z_j(\theta_0
,u,h) - E[Z_j(\theta_0
,u,h)] \right) w(u) du\nonumber  \\
&& - \frac 1{2n} \sum_{j=1}^n \int \dot J(\theta_0
,u,h) (Z_j(\theta_0
,u,h) - E[Z_j(\theta_0
,u,h)]) w(u) du  \nonumber \\
&:=& -\frac 12 (A_n(h)+B_n(h)),
\end{eqnarray}
where terms in $A_n(h)$ and $B_n(h)$ are uncorrelated. On the one hand, we use a multivariate Central Limit Theorem for independent random variables taking values in a Hilbert space, following Kandelaki and Sozanov (1964) or Gikhman and Skorokhod (2004, Theorem 4, page 396).
This will give us the limit behavior of the term
$$
B_n(h)=\frac 1n \sum_{j=1}^n U_j(h), \quad U_j(h):=\int \dot J(\theta_0
,u,h) (Z_j(\theta_0
,u,h) - E[Z_j(\theta_0
,u,h)]) w(u) du.
$$
The random variables $U_j(h)$, $j=1,...,n$ are independent, centered, but their common law depend on $n$ via $h$. Our goal is  to show that
\begin{equation}\label{eq:limitvar}
n h^d Var(B_n(h)) = \sum_{j=1}^n Var\left( \sqrt{\frac{h^d}n}U_j(h) \right) \rightarrow \Sigma, \quad \text{as } n\to \infty
\end{equation}
and that
\begin{equation}\label{eq:Lyapounov}
\sum_{j=1}^n E\left[\left\| \sqrt{\frac{h^d}n} U_j(h)\right\|^4 \right]
= \frac{h^{2d}}n E[\|U_1(h)\|^4] \to 0,\quad \text{as } n\to \infty.
\end{equation}
Indeed, (\ref{eq:Lyapounov}) implies the Lindeberg's condition in  Kandelaki and Sozanov (1964):
$$
\sum_{j=1}^n E \left[\left\|\sqrt{\frac{h^d}n} U_j(h)\right |^2 \cdot \mathbb I_{\left\|\sqrt{h^d/n} U_j(h)\right\| \geq \varepsilon} \right]
\to 0,\quad \text{as } n\to \infty,\text{ for any } \varepsilon >0.
$$
On the other hand, we prove that
\begin{equation}\label{eq:negligeable}
\sqrt{n h^d} A_n(h) \rightarrow 0,   \text{ in probability, }\text{ as } n \to \infty, 
\end{equation}
stating that  $\sqrt{n h^d} A_n(h)$ is a  negligible term  and that, as a consequence,  the limiting  behavior of $\sqrt{nh^d} \dot S_n(\theta_0
)$ is  only driven by $\sqrt{nh^d} B_n(h)$. This will end the proof of the theorem.\\

Let us prove (\ref{eq:limitvar}) and (\ref{eq:Lyapounov}).
Note that $n h^d Var(B_n(h)) = h^d Var(U_1(h))$ and that
\begin{eqnarray*}
&&Var(U_1(h)) \\
&=& \int \int \dot J(\theta_0
,u_1,h) \dot J^\top(\theta_0
,u_2,h) Cov(Z_1(\theta_0
,u_1,h) ,Z_1(\theta_0
,u_2,h) ) w(u_1)w(u_2) du_1 du_2.\\
\end{eqnarray*}
Similarly to Proposition~\ref{mse_contrast}, by uniform continuity in ${\bf x}$ of the integrand in $\dot J$, we get
\begin{equation*}
\lim_{h \to 0} \dot J (\theta_0
,u,h) = \dot J(\theta_0
,u).
\end{equation*}
See that $\|\dot J(\theta_0
,u)\|\leq 2(1+|u|) \|\ell\|_\infty/(1-2P)$ and that the latter upper bound is integrable with respect to the measure $w(u) du$ by assumption on $w$. It remains to study:
\begin{eqnarray*}
&&Cov(Z_1(\theta_0
,u_1,h) ,Z_1(\theta_0
,u_2,h) )\\
& = &  E\left[ Z_1(\theta_0
,u_1,h) Z_1(\theta_0
,u_2,h)\right]
- E\left[ Z_1(\theta_0
,u_1,h)\right] E\left[ Z_1(\theta_0
,u_2,h)\right].
\end{eqnarray*}
From (\ref{expbound}) we deduce that
$$
h^d |E\left[ Z_1(\theta_0
,u_1,h)\right] E\left[ Z_1(\theta_0
,u_2,h)\right]| \leq h^d \left (\frac{4C \int |K|}{1-2P}\right )^2 \to 0,
$$
when  $h\to 0$ as  $n\rightarrow \infty$.
We also have
\begin{eqnarray*}
&& h^d E\left[ Z_1(\theta_0
,u_1,h) Z_1(\theta_0
,u_2,h)\right] \\
&=& \int \int \left( \frac{e^{iu_1 y}}{M(\theta_0
,u_1)} - \frac{e^{-iu_1 y}}{M(\theta_0
,-u_1 )}\right)
\left( \frac{e^{iu_2 y}}{M(\theta_0
,u_2 )} - \frac{e^{-iu_2 y}}{M(\theta_0
,-u_2 )}\right)
\frac 1{h^d} K^2 (\frac{{\bf x} - {\bf x}_0
}h ) g(y,{\bf x}) dy d{\bf x}\\
&=& \int \left( \frac{e^{iu_1 y}}{M(\theta_0
,u_1)} - \frac{e^{-iu_1 y}}{M(\theta_0
,-u_1 )}\right)
\left( \frac{e^{iu_2 y}}{M(\theta_0
,u_2 )} - \frac{e^{-iu_2 y}}{M(\theta_0
,-u_2 )}\right)
g(y,{\bf x}_0
) dy (\int K^2 ) (1+o(1))\\
&=& \int \left( \frac{e^{iu_1 y}}{M(\theta_0
,u_1)} - \frac{e^{-iu_1 y}}{M(\theta_0
,-u_1 )}\right)
\left( \frac{e^{iu_2 y}}{M(\theta_0
,u_2 )} - \frac{e^{-iu_2 y}}{M(\theta_0
,-u_2 )}\right)
g_{{\bf x}_0
}(y) dy \cdot \ell({\bf x}_0
) (\int K^2 ) (1+o(1)),
\end{eqnarray*}
as $h\to 0$.
See also that we can write
\begin{eqnarray*}
  V(\theta_0
,u_1,u_2) &:=& \int \left( \frac{e^{iu_1 y}}{M(\theta_0
,u_1)} - \frac{e^{-iu_1 y}}{M(\theta_0
,-u_1 )}\right)
\left( \frac{e^{iu_2 y}}{M(\theta_0
,u_2 )} - \frac{e^{-iu_2 y}}{M(\theta_0
,-u_2 )}\right)
g_{{\bf x}_0
}(y) dy \\
  &=&  \frac{M(\theta_0
,u_1+u_2)}{M(\theta_0
,u_1) M(\theta_0
,u_2)} f_{{\bf x}_0
}(u_1+u_2)
  -\frac{M(\theta_0
,u_1-u_2)}{M(\theta_0
,u_1) M(\theta_0
,-u_2)} f_{{\bf x}_0
}(u_1-u_2)\nonumber \\
  && -\frac{M(\theta_0
,-u_1+u_2)}{M(\theta_0
,-u_1) M(\theta_0
,u_2)} f_{{\bf x}_0
}(-u_1+u_2)
  +\frac{M(\theta_0
,-u_1-u_2)}{M(\theta_0
,-u_1) M(\theta_0
,-u_2)} f_{{\bf x}_0
}(-u_1-u_2) \nonumber
\end{eqnarray*}
and this is a bounded function with respect to $u_1$ and $u_2$. Therefore
$$
h^d Var(U_1(h)) \rightarrow \int \int \dot J(\theta_0
,u_1) \dot J^\top(\theta_0
,u_2) V(\theta_0
, u_1, u_2) w(u_1) w(u_2) du_1 du_2 =:\Sigma,
$$
as $h\to 0$. This proves (\ref{eq:limitvar}).

Now, denote by $v^{(k)}$ the $k$-th coordinate of a vector $v$ and use Jensen inequality to see that
\begin{eqnarray*}
  E[\|U_1(h)\|^4] & \leq & 3 \left( E[(U_1^{(1)}(h))^4] +E[(U_1^{(2)}(h))^4]+E[(U_1^{(3)}(h))^4] \right)\\
   &\leq & 3\sum_{k=1}^3  E\left[\left( \int \dot J^{(k)} (\theta_0
,u,h)
   (Z_1(\theta_0
,u,h) - E[Z_1(\theta_0
,u,h)])w(u) du\right)^4 \right]\\
   &\leq &  3\sum_{k=1}^3 \int |\dot J^{(k)} (\theta_0
,u,h) |^4 E\left[ |Z_1(\theta_0
,u,h)|^4\right] w(u) du.
\end{eqnarray*}
We have $|\dot J^{(k)} (\theta_0
,u,h) | \leq 4(1+|u|) (\int |K|) \|\ell\|_\infty/(1-2P)^2$ by Lemma~\ref{lemma:ZkJ} and
\begin{eqnarray*}
   E\left[ |Z_1(\theta_0
,u,h)|^4\right]
   &=& \int \int 4 \left |\Im \left(\frac{e^{iuy}}{M(\theta_0
,u)}\right )\right |^4
   \frac 1{h^{4d}} K^4 \left( \frac{{\bf x} - {\bf x}_0
}h\right)
   g(y,{\bf x}) dy d{\bf x} \\
   &\leq & \frac{4}{h^{3d} (1-2P)^4} \int \frac 1{h^d} K^4 \left( \frac{{\bf x} - {\bf x}_0
}h\right)
   \ell({\bf x}) d{\bf x} \\
   &\leq & \frac{O(1)}{h^{3d}} \left(\int K^4\right)   \|\ell\|_\infty,
\end{eqnarray*}
as $h\to 0$. Therefore,
$$
\frac{h^{2d}}n E[\|U_1(h)\|^4] \leq \frac{O(1)}{n h^d} \int |K| \cdot \int K^4 \cdot \int (1+|u|)^4 w(u) du =o(1),
$$
as $n \to \infty$ and $h\to 0$ such that $nh^d \to \infty$. This proves (\ref{eq:Lyapounov}).

\noindent To prove (\ref{eq:negligeable}), we notice that  $A_n(h)$ defined in (\ref{AB}) can be treated similarly to  $T_1$ in (\ref{varT1}). By this remark, we easily prove that $Var(A_n)=o\left((nh^d)^{-1}\right)$  which insure the wanted result.

\bigskip
\noindent  Let us prove that
\begin{eqnarray*}
\ddot S_n(\theta_n){\longrightarrow} \mathcal{I}(\theta_0
),  \text{ in probability, }\text{ as } n \to \infty, 
\end{eqnarray*}
where $\mathcal{I}=\mathcal{I}(\theta_0
)=-\frac 12 \int \dot J (\theta_0
,u) \dot J ^\top (\theta_0
,u) w(u)du$, and  $\dot J (\theta_0
,u)$ is defined in  (\ref{eq:Sigma1}). We start by writing the triangular inequality
$$
\|\ddot S_n(\theta_n) - \mathcal{I}\|\leq \|\ddot S_n(\theta_n) - \ddot S_n(\theta_0
)\| +\|\ddot S_n(\theta_0
)-  E( \ddot S_n(\theta_0
))\|+ \| E( \ddot S_n(\theta_0
))- \mathcal{I}\|.
$$
Then using upper bounds similar to (\ref{c1}) slighly adapted to  $\ddot S_n$ instead of $S_n$ and  the convergence in probability of $\hat \theta_n$ towards  $\theta_0
$ established in Theorem \ref{cons}, we have  that $\|\ddot S_n(\theta_n) - \ddot S_n(\theta_0
)\|\rightarrow 0$ in probability as $n\rightarrow \infty$.
By writting 
\begin{eqnarray*}
E(\ddot S_n(\theta_0
))&=& -\frac 12 \int 
\left (\ddot J(\theta_0
,u,h) J(\theta_0
,u,h) + \dot J(\theta_0
,u,h)  \dot J(\theta_0
,u,h)^\top\right ) w(u)du
\end{eqnarray*}
and noticing, according to Bochner's Lemma,   that  $J(\theta_0
,u,h) \rightarrow 0 $  and $\dot J(\theta_0
,u,h)\rightarrow \dot J(\theta_0
,u) $ as $h\rightarrow 0$, we have, according to the Lebesgue's theorem,
that  $E[\ddot S_n(\theta_0
)]$ tends to $\mathcal{I}$ as $h\rightarrow 0$.
Finally we decompose $-2n(n-1)(\ddot S_n(\theta_0
)-E[\ddot S_n(\theta_0
)])=\sum_{l=1}^3(D_{1,l}+D_{2,l})$ where
\begin{eqnarray*}
D_{1,1}&=&\sum_{k\neq j} \int (\ddot Z_k(\theta_0
,u,h)-\ddot J(\theta_0
,u,h)) (Z_j(\theta,u,h)-J(\theta_0
,u,h))w(u) du  \\
D_{1,2}&=&(n-1) \sum_{k} \int (\ddot Z_k(\theta_0
,u,h)-\ddot J(\theta_0
,u,h)) J(\theta_0
,u,h) w(u) du\\
D_{1,3}&=&(n-1) \sum_{j} \int \ddot J(\theta_0
,u,h) (Z_j(\theta,u,h)-J(\theta_0
,u,h))w(u) du,
\end{eqnarray*}
and
\begin{eqnarray*}
D_{2,1}&=&\sum_{k\neq j} \int (\dot Z_k(\theta_0
,u,h)-\dot J(\theta_0
,u,h)) (\dot Z_j(\theta,u,h)-\dot J(\theta_0
,u,h))^\top w(u) du  \\
D_{2,2}&=&(n-1) \sum_{k} \int (\dot Z_k(\theta_0
,u,h)-\dot J(\theta_0
,u,h)) J(\theta_0
,u,h)^\top w(u) du\\
D_{2,3}&=&(n-1) \sum_{j} \int \dot J(\theta_0
,u,h) (Z_j(\theta,u,h)-J(\theta_0
,u,h))^\top w(u) du.
\end{eqnarray*}
Noticing that  terms $D_{i,3}$, $i=1,2$, respectively  $D_{i,j,}$, $i=1,2$ and $j=2,3$,  can be treated as $T_1$ respectively  $T_2$ in the proof of Proposition \ref{mse_contrast}, we obtain 
$$
Var\left(\ddot S_n(\theta_0
)\right)=O\left(\frac{1}{nh^d}\right) ,
$$
which concludes the proof.
\end{proof}

\noindent {\bf Aknowledgements.} The authors  thank  warmly Dr.'s  Bowen and  Chappell for providing  the Positron Emission Tomography dataset presented
in Bowen et al. (2012, Fig. 4), as well as   Dr. Wang  for sharing  the EM-type algorithm code developed in Huang et al. (2013).

\end{document}